\documentclass[hidelinks,onefignum,onetabnum]{siamart251216}

\usepackage{amsfonts}
\usepackage{graphicx}
\usepackage{epstopdf}
\usepackage{booktabs}
\usepackage{enumitem}
\usepackage{mathtools}
\usepackage[normalem]{ulem}
\usepackage{algorithm}
\usepackage{algorithmic}
\usepackage{float} 
\usepackage{cite}  

\graphicspath{{images/}}
\ifpdf
  \DeclareGraphicsExtensions{.eps,.pdf,.png,.jpg}
\else
  \DeclareGraphicsExtensions{.eps}
\fi

\providecommand{\qedhere}{}

\newsiamremark{remark}{Remark}
\newsiamremark{example}{Example}
\crefname{example}{Example}{Examples}

\newcommand{\R}{\mathbb{R}}
\newcommand{\ip}[2]{\left\langle #1,\, #2 \right\rangle}
\newcommand{\grad}{\nabla}

\newcommand{\Hsav}{H_{\mathrm{SAV}}}
\newcommand{\Hdg}{H_{\mathrm{DG}}}
\newcommand{\Hhat}{\widehat{H}}
\newcommand{\Esav}{E_{\mathrm{SAV}}}
\newcommand{\Edg}{E_{\mathrm{DG}}}
\newcommand{\Ehat}{\widehat{E}}

\headers{A Unified Discrete Gradient--SAV Framework}{E.~Celledoni, D.~Mart\'in de Diego, B.~Owren, and M.~Vaquero}

\title{A Unified Discrete Gradient--SAV Framework for Structure-Preserving Integration}

\author{
  Elena Celledoni\thanks{NTNU, Trondheim, Norway (\email{elena.celledoni@ntnu.no}).}
  \and David Mart\'in de Diego\thanks{ICMAT (CSIC), Madrid, Spain (\email{david.martin@icmat.es}).}
  \and Brynjulf Owren\thanks{NTNU, Trondheim, Norway (\email{brynjulf.owren@ntnu.no}).}
  \and Miguel Vaquero\thanks{IE University, Madrid, Spain (\email{miguel.vaquero@ie.edu}).}
}

\ifpdf
\hypersetup{
  pdftitle={A Unified Discrete Gradient--SAV Framework for Structure-Preserving Integration},
  pdfauthor={E. Celledoni, D. Mart\'in de Diego, B. Owren, M. Vaquero}
}
\fi

\begin{document}

\maketitle

\begin{abstract}
{
We present a framework combining discrete gradient (DG) methods with the Scalar Auxiliary Variable (SAV) approach to construct structure-preserving integrators for dissipative and conservative systems. The key observation is that SAV quadratization lifts the dynamics to an extended state space on which the modified energy has an exact discrete-gradient identity. This viewpoint yields three integrators with different accuracy and cost profiles: a first-order semi-implicit Forward Euler scheme, a second-order self-adjoint Midpoint scheme, and a second-order Predictive scheme with reduced implicitness. The construction extends to almost-Poisson systems and preserves selected Casimir invariants under an enforceable discrete condition. Numerical experiments cover the Allen--Cahn equation, an Ohta--Kawasaki-type nonlocal gradient flow, a double-well Hamiltonian oscillator, and a Poisson system with a nonlinear cubic Casimir.
}
\end{abstract}



\section{Introduction}
\label{sec:introduction}

\subsection{Motivation: Two paradigms for discrete energy laws}
\label{sec:dg-sav-related}
{
Reliable simulations of dissipative and Hamiltonian systems require integrators that respect their geometric structure: otherwise, even accurate schemes can distort energy behavior and long-time qualitative dynamics~\cite{HairerLubichWanner2006,SanzSernaCalvo1994,LeimkuhlerReich2004}. Symplectic methods control long-time energy errors in conservative problems, while preserving the free-energy law is essential in phase-field and thin-film models, coarsening dynamics, and Wasserstein gradient flows~\cite{ElliottStuart1993,DuNicolaides1991,Eyre1998}. Two complementary paradigms address these requirements. Discrete gradient (DG) methods enforce exact discrete chain rules, usually through nonlinear solves, whereas scalar auxiliary variable (SAV) methods obtain unconditional modified-energy stability through inexpensive constant-coefficient linear solves. We combine them by applying a DG chain rule to the SAV-quadratized extended energy.
}

\paragraph{Continuous energy laws}
The systems we treat have either dissipative or conservative energy laws:
\begin{align}
\dot{x}&=-M(x)\nabla E(x),\qquad M(x)\succeq 0,\label{eq:gradflow-cont}\\
\dot{x}&=J(x)\nabla H(x),\qquad J(x)^\top=-J(x).\label{eq:poisson-cont}
\end{align}
The chain rule
\begin{equation}\label{eq:chainrule-cont}
\frac{\mathrm d}{\mathrm dt}E(x(t))=\langle \nabla E(x(t)),\,\dot{x}(t)\rangle= \langle \nabla E, S\,\nabla E\rangle
\end{equation}
reduces the energy rate to a quadratic form $\langle \nabla E, S\,\nabla E\rangle$ in the structure matrix~$S$: $S=-M\preceq 0$ gives dissipation, $S=J=-J^\top$ gives exact conservation. A structure-preserving integrator must enforce the same identity at the discrete level.

\paragraph{Discrete gradient (DG) methods}
{
DG methods replace $\nabla E$ by a two-point map satisfying an exact discrete chain rule, the finite-difference counterpart of~\eqref{eq:chainrule-cont} formalized in \cref{def:discrete-gradient}~\cite{Gonzalez1996,McLachlanQuispelRobidoux1999,QuispelMcLaren2008}. Combined with a two-point approximation of $M$ or $J$ that retains the relevant structural property, the same quadratic-form argument gives $E(x^{n+1})\le E(x^n)$ or $H(x^{n+1})=H(x^n)$; see \cref{lem:energy-law}. Exact treatment of the full energy, however, generally couples the discrete gradient to the unknown $x^{n+1}$ and may require an expensive nonlinear solve~\cite{Gonzalez1996,McLachlanQuispelRobidoux1999}.
}

\paragraph{Scalar Auxiliary Variable (SAV) methods}
{
The SAV methodology prioritizes computational efficiency~\cite{ShenXuYang2018,Shen2018,ShenXuYang2019}. In its classical gradient-flow form, one writes
}
\[
  E=E_{\mathrm{lin}}+E_{\mathrm{nl}},
\]
{
where $E_{\mathrm{lin}}$ is typically quadratic and admits a simple constant-coefficient implicit treatment, while $E_{\mathrm{nl}}$ contains the nonlinear remainder. Introducing $r^2=E_{\mathrm{nl}}(x)+C_0$ and evaluating the nonlinear SAV coefficients explicitly reduces each step to a constant-coefficient linear solve; the standard SAV/CN and SAV/BDF2 formulas are recalled in Appendix~\ref{sec:supp-sav-discrete}. These schemes dissipate a modified energy unconditionally, although that quantity need not agree with the original energy away from the constraint manifold~\cite{ShenXuYang2018,Shen2018}. Conservative extensions require additional care because frozen coefficients must retain the skew-symmetry cancellation $\langle\nabla H,J\nabla H\rangle=0$; see \cref{rem:sav-beyond-gradient-flows}.
}

\paragraph{A unifying viewpoint: SAV as quadratization, DG--SAV as DG on the quadratized system}
{
The DG--SAV construction starts from a decomposition of the energy,
}
\[
  E=\Edg+\Esav,
\]
{
where $\Edg$ is the component treated by a discrete gradient and
$\Esav$ is the component handled through the scalar auxiliary variable.
SAV quadratization then lifts the dynamics from $x$ to $(x,r)$ by setting
}
\[
  r^2=\Esav(x)+C_0,
\]
{
so that the corresponding modified energy is
}
\[
  \widehat E(x,r)=\Edg(x)+r^2.
\]
{
The auxiliary contribution is quadratic and satisfies the elementary discrete chain rule, with $\bar r:=\tfrac12(r^{n+1}+r^n)$,
}
\[
  (r^{n+1})^2-(r^n)^2=2\bar r\,(r^{n+1}-r^n).
\]
{
Hence the natural extended discrete gradient is
}
\[
  \bigl(\widetilde{\nabla}\Edg(x^{n+1},x^n),\, r^{n+1}+r^n\bigr).
\]
{
Eliminating the $r$-increment yields a generalized gradient $\eta$ that combines $\widetilde{\nabla}\Edg$ with the SAV contribution. Conservation or dissipation of $\widehat E$ is therefore the usual discrete-gradient algebra on the quadratized extended system, even when the SAV coefficients are frozen at explicitly computable states. This extended-system interpretation is shared by Lu, Wang, and Sun~\cite{LuWangSun2025}, who recover EQ, SAV, and QAV schemes from a dimension-extended linear-gradient formulation; \cref{sec:continuous-tensor-form} identifies the corresponding structure tensor precisely.

The split $E=\Edg+\Esav$ is the central design choice. Terms retained in $\Edg$ remain tied directly to the original energy but may increase the cost of the implicit solve; terms assigned to $\Esav$ are cheaper to treat, but only the modified energy $\widehat E=\Edg+r^2$ is preserved exactly away from the constraint manifold. Thus DG--SAV interpolates between pure DG and SAV-type methods. In gradient flows, $\Edg$ may be nonlinear rather than restricted to the classical quadratic $E_{\mathrm{lin}}$. The structural identification is formalized in \cref{thm:modified-energy-evolution} and then extended to dissipative, Hamiltonian, and Poisson systems.
}

\paragraph{Related work}
{
Discrete gradient methods originate with Gonzalez~\cite{Gonzalez1996} and the systematic development of McLachlan, Quispel, and Robidoux~\cite{McLachlanQuispelRobidoux1999}, building on the coordinate-increment construction of Itoh and Abe~\cite{ItohAbe1988}. The Average Vector Field (AVF) method~\cite{QuispelMcLaren2008} is the canonical symmetric discrete gradient and the minimal-stage energy-preserving Runge--Kutta method for polynomial Hamiltonians~\cite{CelledoniOwrenSun2014}. Related time finite-element and Petrov--Galerkin constructions were developed by Betsch and Steinmann~\cite{BetschSteinmann2000JCP,BetschSteinmann2000NBody}. Subsequent developments include discrete variational derivatives for PDEs~\cite{FurihataMatsuo2010}, systematic AVF discretizations of Hamiltonian and gradient PDEs~\cite{CelledoniGrimmMcLachlan2012}, methods on Riemannian manifolds~\cite{CelledoniEidnesOwrenRingholm2020}, preservation of multiple first integrals~\cite{DahlbyOwrenYaguchi2011,NortonMcLarenQuispelSternZanna2013}, dissipative and optimization dynamics~\cite{RiisEhrhardtQuispelSchonlieb2022,EhrhardtRiisRingholmSchonlieb2024}, B-series and order theory~\cite{CelledoniMcLachlanOwrenQuispel2010,Eidnes2022}, and linear energy-preserving integrators for Poisson systems~\cite{CohenHairer2011}. Relaxation and projection methods provide a complementary high-order route by correcting a standard step through a scalar or low-dimensional condition; applications include Hamiltonian systems and dispersive PDEs with simultaneous conservation of mass, momentum, and energy~\cite{RanochaKetcheson2020Hamiltonian,RanochaMitsotakisKetcheson2021,RanochaKetcheson2025NLS,RanochaKetcheson2025BBMKdVNLS}.

The SAV approach was introduced for gradient flows by Shen, Xu, and Yang~\cite{ShenXuYang2018}; see~\cite{Shen2018} for analysis and~\cite{ShenXuYang2019} for a survey. It refines invariant energy quadratization~\cite{Yang2017IEQ}, and extensions to conservative systems include~\cite{KemmochiSato2021}. Related bridges between DG and SAV include SAV algorithms for discrete-gradient optimization systems~\cite{LiuShenZhang2023} and the explicit modified-energy-conserving methods of Bilbao, Ducceschi, and Zama for separable Hamiltonians~\cite{Bilbao2023,ZamaDucceschiBilbao2023}. Closest to our structural viewpoint is Lu, Wang, and Sun~\cite{LuWangSun2025}: their dimension-extended linear-gradient formulation inherits energy laws, Poisson structure, and Casimirs, recovers EQ, SAV, and QAV schemes, and yields fully implicit extended discrete-gradient methods that preserve the auxiliary constraint and hence the original energy. Our construction uses the same extended-system interpretation but freezes SAV coefficients at explicitly computable reference states. It thereby retains semi-implicit, or in the separable Hamiltonian case fully explicit, updates while preserving a modified energy for an arbitrary split; in the Poisson setting it targets selected Casimirs of the original tensor rather than the constraint Casimir of the extension. The UEPI framework~\cite{CelledoniOwrenXuShenYaguchi2025UEPI} instead learns problem-adapted energy-behavior-preserving integrators from data; here we give a closed analytic construction in which the split controls the cost--structure trade-off.
}

\subsection{Main Contributions}

The principal contributions are:
{
\begin{itemize}
  \item An extended-system identity showing that DG--SAV is the discrete-gradient chain rule for $\Ehat=\Edg+r^2$, rewritten through $\eta$ after the $r$-increment is eliminated.
 \item A tunable split between pure DG and SAV-type quadratization; for separable Hamiltonian systems, the framework yields an explicit family of modified-energy-conserving methods similar to those of Bilbao, Ducceschi, and Zama~\cite{Bilbao2023,ZamaDucceschiBilbao2023}.
  \item Three variants---first-order semi-implicit Forward Euler, second-order self-adjoint Midpoint, and second-order Predictive---with exact modified-energy laws and local well-posedness and accuracy results.
  \item A Poisson extension that preserves the modified energy and selected Casimirs through a projected discrete tensor condition; see \cref{sec:poisson}.
  \item Numerical evidence that, in the nonlocal $p=6$ benchmark, Predictive and Midpoint DG--SAV attain Full DG accuracy at lower cost---roughly a factor of two for Predictive and $1.2$--$1.5$ for Midpoint over the resolved work--precision range---and, for the respective energy splits used in this benchmark, produce absolute SAV gaps six to seven orders of magnitude smaller than those of the classical SAV schemes.
\end{itemize}
}

\subsection{Paper Organization}

{
\Cref{sec:preliminaries} fixes notation. \Cref{sec:unified-principle,sec:variants,sec:poisson,sec:separable} develop the DG--SAV principle, its three variants, the Poisson extension, and the separable Hamiltonian case. \Cref{sec:numerics} presents the numerical experiments. Additional details are provided in the appendices.
}

\section{Preliminaries and Notation}
\label{sec:preliminaries}

\subsection{Notation}

We work in finite-dimensional Euclidean space $\R^d$ equipped with the standard inner product $\ip{u}{v} = \sum_{i=1}^d u_i v_i$ and norm $\|u\| = \sqrt{\ip{u}{u}}$. For PDE problems, $\R^d$ represents the space of degrees of freedom after spatial discretization (e.g., Fourier modes or finite-element coefficients), and $\ip{\cdot}{\cdot}$ is the corresponding discrete $L^2$ inner product.

Given a smooth scalar function $F:\R^d \to \R$, we write $\nabla F(x) \in \R^d$ for its gradient, defined by $\ip{\nabla F(x)}{v} = \mathrm{d}F(x)[v]$ for all $v \in \R^d$.

Time is discretized uniformly with step size $h > 0$, giving $t_n = nh$ and $x^n \approx x(t_n)$. Midpoint averages are denoted $\bar{x}=\tfrac{1}{2}(x^{n+1}+x^n)$ and $\bar{r}=\tfrac{1}{2}(r^{n+1}+r^n)$. Discrete gradients (two-point maps approximating the continuous gradient; see Section~\ref{sec:discrete-gradients}) are written $\widetilde{\nabla}E(x^{n+1}, x^n)$.

The energy functional is generically denoted $E$ for dissipative systems and $H$ for Hamiltonian/Poisson systems. The state variable is $x$ in the dissipative setting and $z$ (or $(q,p)$ for separable Hamiltonians) in the conservative setting. When the DG--SAV splitting is in effect, we write $E = \Edg + \Esav$ (respectively $H = \Hdg + \Hsav$), and the scalar auxiliary variable is always denoted~$r$.

A one-step method $x^{n+1} = \Phi_h(x^n)$ is said to be of \emph{order~$p$} if, for every sufficiently smooth solution $x(t)$ of the ODE, the local truncation error satisfies $\|x(t+h) - \Phi_h(x(t))\| = O(h^{p+1})$ as $h \to 0$. Equivalently, the global error over a fixed time interval satisfies $\|x^n - x(t_n)\| = O(h^p)$.

\subsection{Scalar Auxiliary Variable (SAV) Method}

The Scalar Auxiliary Variable (SAV) method was introduced for gradient flows~\cite{ShenXuYang2018} but extends beyond this setting. To fix the notation used below, consider the gradient system
\begin{equation}\label{eq:gradient-system}
  \dot{x} = -\grad E(x),
\end{equation}
where $E:\R^d\to\R$ is an energy functional. This is the $M\equiv I$ case of~\eqref{eq:gradflow-cont}, used here for consistency with the original SAV literature~\cite{ShenXuYang2018,ShenXuYang2019}; the general $M(x)$-dependent case is treated from Section~\ref{sec:unified-principle} onward.

\paragraph{Energy splitting}
The SAV approach is built on a decomposition of the energy into a quadratic part (whose gradient is linear) and a nonlinear part:
\begin{equation}\label{eq:sav-lin-nl-split}
  E(x) = E_{\mathrm{lin}}(x) + E_{\mathrm{nl}}(x),
\end{equation}
where $E_{\mathrm{lin}}$ admits a simple implicit treatment (typically $E_{\mathrm{lin}}(x)=\frac12 x^\top A x$ with constant symmetric positive-semidefinite $A$) and $E_{\mathrm{nl}}$ contains the nonlinear terms.

A scalar auxiliary variable $r(t)$ is introduced to track only the nonlinear part:
\[
  r^2 = E_{\mathrm{nl}}(x) + C_0,
\]
where $C_0$ is chosen so that $E_{\mathrm{nl}}(x)+C_0 \ge \delta > 0$ on the relevant solution region. Differentiating along solutions yields $2r\,\dot{r} = \ip{\grad E_{\mathrm{nl}}}{\dot{x}}$. The gradient system is reformulated as:
\begin{equation}
  \label{eq:sav-continuous-prelim}
  \begin{aligned}
    \dot{x} &= -\grad E_{\mathrm{lin}}(x) - \frac{r}{\sqrt{E_{\mathrm{nl}}(x)+C_0}}\,\grad E_{\mathrm{nl}}(x), \\[4pt]
    \dot{r} &= \frac{\ip{\grad E_{\mathrm{nl}}(x)}{\dot{x}}}{2\sqrt{E_{\mathrm{nl}}(x)+C_0}}.
  \end{aligned}
\end{equation}
On the constraint manifold $r = \sqrt{E_{\mathrm{nl}}(x)+C_0}$, this reduces to the original gradient flow.

\paragraph{SAV discretization and modified energy dissipation}
{
The two standard second-order choices are SAV/CN (Crank--Nicolson) and SAV/BDF2~\cite{ShenXuYang2019}. Both treat $E_{\mathrm{lin}}$ implicitly and evaluate the nonlinear coefficient at an explicit $O(h^2)$ extrapolant, producing a constant-coefficient linear system and unconditional stability for their respective modified energies; see~\cite[Theorems~2.1--2.2]{ShenXuYang2019}. Their formulas are given in Appendix~\ref{sec:supp-sav-discrete}.
}


\begin{remark}[SAV beyond gradient flows]\label{rem:sav-beyond-gradient-flows}
{
SAV was developed for dissipative gradient flows, where the objective is unconditional modified-energy stability. Conservative extensions must retain the skew-symmetry cancellation after introducing the auxiliary variable. For separable Hamiltonians, Bilbao, Ducceschi, and Zama~\cite{Bilbao2023} obtained explicit schemes that conserve a modified energy exactly; see also~\cite{ZamaDucceschiBilbao2023} on the shift constant. Kemmochi and Sato~\cite{KemmochiSato2021} treated conservative systems with unbounded energies. Lu, Wang, and Sun~\cite{LuWangSun2025} discretized a dimension-extended reformulation to obtain fully implicit discrete-gradient schemes that conserve the original, possibly nonpolynomial, Hamiltonian. DG--SAV treats the dissipative and conservative settings within one framework; see \cref{sec:variants}.
}
\end{remark}

\subsection{Discrete Gradients}
\label{sec:discrete-gradients}

\begin{definition}[Discrete Gradient]\label{def:discrete-gradient}
Let $E:\R^d \to \R$ be continuously differentiable. A \emph{discrete gradient} of $E$ is a map
\[
  \widetilde{\nabla} E : \R^d \times \R^d \to \R^d, \qquad (x^+,x^-)\mapsto \widetilde{\nabla} E(x^+,x^-),
\]
satisfying:
\begin{enumerate}[label=(\roman*)]
  \item \textbf{Discrete chain rule:}
    \begin{equation}\label{eq:dg-identity}
      E(x^+) - E(x^-) = \ip{\widetilde{\nabla} E(x^+,x^-)}{x^+ - x^-};
    \end{equation}
  \item \textbf{Consistency:}
    \begin{equation}\label{eq:dg-consistency}
      \widetilde{\nabla} E(x,x) = \grad E(x).
    \end{equation}
\end{enumerate}
\end{definition}

The two discrete gradients used in this paper are the symmetric Average Vector Field (AVF) gradient~\cite{QuispelMcLaren2008}
\begin{equation}\label{eq:avf-dg}
  \widetilde{\nabla}_{\mathrm{AVF}} E(x^+,x^-)
  = \int_0^1 \nabla E\big((1-\xi)x^- + \xi x^+\big)\,d\xi
\end{equation}
and the Gonzalez gradient~\cite{Gonzalez1996}, $\widetilde{\nabla}_{\mathrm{G}} E(x^+,x^-) = \nabla E(\bar x) + \bigl[E(x^+) - E(x^-) - \langle\nabla E(\bar x), x^+ - x^-\rangle\bigr](x^+ - x^-)/\|x^+ - x^-\|^2$, with $\bar x = \tfrac12(x^++x^-)$. Both are symmetric in their arguments, ensuring time-reversibility when used in self-adjoint schemes; the coordinate-increment construction of Itoh--Abe~\cite{ItohAbe1988} is an alternative not used here.

\paragraph{Energy law template}
The immediate payoff of the discrete chain rule is that energy conservation or dissipation reduces to a one-line algebraic argument---the discrete analog of the continuous mechanism described in Section~\ref{sec:dg-sav-related}.

\begin{lemma}[Energy Law Template]\label{lem:energy-law}
Let $S \in \R^{d\times d}$ be a structure matrix and let $\eta \in \R^d$. Consider a one-step scheme of the form
\[
  \frac{x^{n+1} - x^n}{h} = S\,\eta.
\]
Then the energy change is governed by
\[
  E(x^{n+1}) - E(x^n)
  = \ip{\widetilde{\nabla} E(x^{n+1},x^n)}{x^{n+1} - x^n}
  = h\,\ip{\widetilde{\nabla} E(x^{n+1},x^n)}{S\,\eta}.
\]
Two fundamental cases arise:
\begin{enumerate}[label=(\roman*)]
  \item \textbf{Conservation (skew-symmetric $S$):} If $S^\top = -S$ and $\eta = \widetilde{\nabla} E(x^{n+1},x^n)$, then
  \[
    E(x^{n+1}) - E(x^n)
    = h\,\ip{\widetilde{\nabla} E(x^{n+1},x^n)}{S\,\widetilde{\nabla} E(x^{n+1},x^n)}
    = 0.
  \]
  \item \textbf{Dissipation (negative semidefinite $S$):} If $S = -M$ with $M$  positive semidefinite and $\eta = \widetilde{\nabla} E(x^{n+1},x^n)$, then
  \[
    E(x^{n+1}) - E(x^n)
    = -h\,\ip{\widetilde{\nabla} E(x^{n+1},x^n)}{M\,\widetilde{\nabla} E(x^{n+1},x^n)}
    \le 0.
  \]
\end{enumerate}
\end{lemma}

\begin{proof}
The first identity is the discrete chain rule followed by substitution of the scheme; the conservation and dissipation statements then follow from $\ip{v}{Sv}=0$ for skew-symmetric $S$ and $\ip{v}{Mv}\ge0$ for $M\succeq0$.
\end{proof}

This template forms the basis for all energy-preserving and energy-dissipating schemes developed in this paper. In the DG--SAV framework of Section~\ref{sec:unified-principle}, the vector $\eta$ will combine a discrete gradient of $\Edg$ with SAV-mediated terms from $\Esav$, and the same algebraic argument will yield exact conservation or dissipation of a \emph{modified} energy.

\paragraph{Midpoint consistency}
The midpoint DG--SAV schemes developed in Sections~\ref{sec:dissipative}--\ref{sec:poisson} require the following approximation property, which holds for symmetric discrete gradients.

\begin{lemma}[Midpoint Consistency of Symmetric Discrete Gradients]\label{lem:symmDG-midpoint}
Let $E \in C^3(\R^d)$ and let $\widetilde{\nabla} E$ be a \emph{symmetric} discrete gradient, i.e., $\widetilde{\nabla} E(x^+,x^-) = \widetilde{\nabla} E(x^-,x^+)$. Assume furthermore that the map $(x^+,x^-) \mapsto \widetilde{\nabla} E(x^+,x^-)$ is $C^2$ in a neighborhood of the diagonal $\{(x,x) : x \in \R^d\}$. Then for $x^+, x^-$ sufficiently close and $\bar{x} = \frac{1}{2}(x^+ + x^-)$,
\begin{equation}\label{eq:symmDG-midpoint}
  \widetilde{\nabla} E(x^+,x^-) = \nabla E(\bar{x}) + O(\|x^+ - x^-\|^2).
\end{equation}
\end{lemma}

\begin{proof}
Write $x^\pm = \bar{x} \pm \frac{1}{2}\delta$ where $\delta = x^+ - x^-$. Define $\phi(\delta) := \widetilde{\nabla} E(\bar{x} + \frac{1}{2}\delta, \bar{x} - \frac{1}{2}\delta)$. By the $C^2$ assumption, $\phi$ admits a Taylor expansion about $\delta = 0$:
\[
  \phi(\delta) = \phi(0) + D\phi(0)\,\delta + O(\|\delta\|^2).
\]
At $\delta = 0$, consistency of the discrete gradient gives $\phi(0) = \widetilde{\nabla} E(\bar{x}, \bar{x}) = \nabla E(\bar{x})$. By symmetry, $\widetilde{\nabla} E(x^+,x^-) = \widetilde{\nabla} E(x^-,x^+)$, so $\phi(\delta) = \phi(-\delta)$ for all $\delta$, i.e., $\phi$ is an even function. Differentiating at $\delta = 0$: $D\phi(0) = -D\phi(0)$, hence $D\phi(0) = 0$. Therefore
\[
  \widetilde{\nabla} E(x^+,x^-) = \phi(\delta) = \nabla E(\bar{x}) + O(\|\delta\|^2). \qedhere
\]
\end{proof}

\section{Unified DG--SAV Principle}
\label{sec:unified-principle}

This section builds the unified DG--SAV template on which the dissipative, Hamiltonian, and Poisson constructions of Sections~\ref{sec:dissipative}--\ref{sec:poisson} all rest; its central structural identification is stated in \cref{thm:modified-energy-evolution}.

\subsection{Extended System and Modified Energy}

We apply the construction to the dissipative and conservative systems introduced in~\eqref{eq:gradflow-cont}--\eqref{eq:poisson-cont}, written uniformly as $\dot x=S(x)\nabla E(x)$ with $S=-M\preceq0$ in the dissipative case and, after replacing $E$ by $H$, $S=J=-J^\top$ in the Hamiltonian/Poisson case.
The starting point is to decompose the system energy as
\begin{equation}\label{eq:energy-split}
  E = \Edg + \Esav,
\end{equation}
where $\Edg$ is the part for which a discrete gradient can be evaluated efficiently (e.g., quadratic), and $\Esav$ is the computationally expensive part to be handled via an auxiliary scalar.
This decomposition applies equally to dissipative systems $\dot{x} = -M(x)\,\nabla E(x)$ ($M \succeq 0$) and Hamiltonian systems $\dot{x} = J(x)\,\nabla H(x)$ ($J^\top = -J$), with $E$ replaced by the Hamiltonian~$H$ in the latter case.

{Following the SAV methodology (Section~\ref{sec:preliminaries}), we fix a constant $C_0$ and a lower bound $\delta>0$ such that $\Esav(x)+C_0 \ge \delta$ in the region of interest, and introduce the scalar auxiliary variable}
\begin{equation}\label{eq:r-constraint}
  r = \sqrt{\Esav(x) + C_0},
\end{equation}
which defines a constraint manifold in the extended $(x,r)$ space. On this manifold, the \emph{modified energy}
\begin{equation}\label{eq:modified-energy}\tag{ME}
  \Ehat(x,r) := \Edg(x) + r^2
\end{equation}
satisfies $\Ehat(x,r) = E(x) + C_0$, so it differs from the original energy only by the constant~$C_0$. With the shorthand $\Sigma(x) := \sqrt{\Esav(x)+C_0}$ and $G(x) := \grad\Esav(x)$, the dynamics $\dot x = S(x)\nabla E(x)$ (with $S(x) = -M(x)$, $M\succeq 0$, in the dissipative case and $S(x)^\top = -S(x)$ in the conservative case) extend on $(x,r)$ as
\begin{equation}\label{eq:DG-SAV-continuous-general}
  \dot x = S(x)\bigl(\grad\Edg(x) + (r/\Sigma(x))\,G(x)\bigr),
  \qquad
  \dot r = \frac{\langle G(x),\,\dot x\rangle}{2\Sigma(x)};
\end{equation}
on the constraint manifold $r = \Sigma(x)$ the original system $\dot x = S(x)\nabla E(x)$ is recovered.

\paragraph{Continuous tensor form}
\label{sec:continuous-tensor-form}
Set $y=(x,r)$ and $\nabla_y\Ehat(y)=(\grad\Edg(x),2r)^\top$. The aim is to write the lifted dynamics as a single structure tensor acting on $\nabla_y\Ehat$, so that dissipation or conservation is visible before discretization. In the dissipative case ($S=-M$, $M\succeq0$), introduce the \emph{augmented dissipation tensor}
\begin{equation}\label{eq:augmented-dissipation-tensor}
  \widetilde{M}(x) :=
  \begin{pmatrix}
    M(x) & \dfrac{M(x)\,G(x)}{2\,\Sigma(x)} \\[6pt]
    \dfrac{G(x)^\top M(x)}{2\,\Sigma(x)} & \dfrac{G(x)^\top M(x)\,G(x)}{4\,\Sigma(x)^2}
  \end{pmatrix},
\end{equation}
where $G(x)=\grad\Esav(x)$ and $\Sigma(x)=\sqrt{\Esav(x)+C_0}$. The factorization
\begin{equation}\label{eq:dissipation-factor}
  \widetilde{M}(x) = \begin{pmatrix} I \\[2pt] G(x)^\top/(2\Sigma(x)) \end{pmatrix} M(x) \begin{pmatrix} I & G(x)/(2\Sigma(x)) \end{pmatrix}
\end{equation}
shows $\widetilde{M}(x) \succeq 0$. In the conservative case ($S=J$, $J^\top=-J$), use instead the augmented skew tensor
\begin{equation}\label{eq:augmented-skew-tensor}
  K(x) := \begin{pmatrix} J(x) & J(x)G(x)/(2\Sigma(x)) \\[2pt] -(J(x)G(x))^\top/(2\Sigma(x)) & 0 \end{pmatrix}.
\end{equation}
This $K(x)$ is skew-symmetric, with off-diagonal blocks that are negative transposes. With these definitions, the lifted dynamics~\eqref{eq:DG-SAV-continuous-general} have the compact extended-tensor form
\begin{equation}\label{eq:extended-tensor-dynamics}
  \dot y = \mathcal{S}_{\mathrm{ext}}(x)\,\nabla_y\Ehat(y),
  \qquad
  \mathcal{S}_{\mathrm{ext}}(x)=
  \begin{cases}
    -\widetilde{M}(x), & S=-M,\\
    K(x), & S=J.
  \end{cases}
\end{equation}
The negative semidefiniteness or skew-symmetry of $\mathcal{S}_{\mathrm{ext}}$ gives the continuous modified-energy law $\tfrac{d}{dt}\Ehat = \langle \eta, S(x)\,\eta\rangle$ with $\eta := \grad\Edg + (r/\Sigma)\,G$ (in conservative notation, $\tfrac{d}{dt}\Hhat=0$). The DG--SAV discretization below is motivated by applying a discrete-gradient replacement to these extended-tensor dynamics; its energy identity is the exact stepwise counterpart of the continuous law. For the scalar quadratization $r=\Sigma(x)$, the tensors \eqref{eq:augmented-dissipation-tensor} and \eqref{eq:augmented-skew-tensor} coincide with the extended structure matrix of Lu, Wang, and Sun~\cite[Example~3.1]{LuWangSun2025}, who develop such extensions for general vector-valued dimension-raising functions and prove that they inherit dissipativity, Poisson structure, and Casimirs.

\paragraph{{Choice of the split and the SAV gap}}
{
The split $E = \Edg + \Esav$ is a design choice that balances the cost of the discrete-gradient solve against the drift from the constraint manifold, measured by the \emph{SAV gap} $|r^2-(\Esav(x)+C_0)|$. Assigning more of the energy to $\Edg$ generally reduces this gap but may require a more expensive nonlinear solve; the endpoint $\Esav=0$ is pure DG. The shift $C_0$ is chosen so that $\Esav+C_0$ remains bounded away from zero both along the trajectory and in the local neighborhood sampled by the method, ensuring that $r=\sqrt{\Esav+C_0}$ and the coefficients containing $1/\Sigma$ are well-defined. Its effect on the discrete SAV gap is analyzed in~\cite{ZamaDucceschiBilbao2023}. In the extended-system terminology of~\cite{LuWangSun2025}, $r-\sqrt{\Esav(x)+C_0}$ is the constraint Casimir. The SAV gap measures departure from its zero level set, equivalently from the constraint manifold $r=\sqrt{\Esav(x)+C_0}$. Enforcing that Casimir exactly recovers the original energy law but leads to a fully implicit solve of pure-DG type; the variants below instead accept a small nonzero gap in exchange for frozen coefficients and cheaper steps.
}

\subsection{General One-Step DG--SAV Template and Energy Law}
\label{sec:general-template}

Motivated jointly by the extended tensor dynamics~\eqref{eq:extended-tensor-dynamics} and by the discrete-gradient chain-rule mechanism of Section~\ref{sec:discrete-gradients}, the DG--SAV discretization replaces $\nabla_y\Ehat$ by an extended discrete gradient and evaluates the tensor coefficients at a chosen reference state. Writing this extended update in $(x,r)$ components and collecting the SAV contribution into the vector $\eta$ gives the following one-step template. Let $S_*$ be a structure operator (skew-symmetric for conservation, $S_* = -M_*$ with $M_* \succeq 0$ for dissipation). The general DG--SAV scheme reads:
\begin{equation}\label{eq:general-dgsav}\tag{DG-SAV}
  \begin{aligned}
    \frac{x^{n+1} - x^n}{h} &= S_*\left(\widetilde{\nabla}\Edg(x^{n+1},x^n) + \frac{\bar{r}}{\Sigma_*}\,G_*\right), \\[6pt]
    r^{n+1} - r^n &= \frac{\ip{G_*}{x^{n+1} - x^n}}{2\Sigma_*},
  \end{aligned}
\end{equation}
where $\widetilde{\nabla}\Edg(x^{n+1},x^n)$ is a discrete gradient of $\Edg$, the quantities $\Sigma_* \approx \sqrt{\Esav(x) + C_0}$ and $G_* \approx \grad\Esav(x)$ are SAV coefficients evaluated at an appropriate reference state $x_*$ (the specific choice---current, midpoint, or predicted---defines the three scheme variants in Section~\ref{sec:variants}), and $\bar{r} = \frac{1}{2}(r^{n+1} + r^n)$.

The design of the $r$-equation is not arbitrary: the factor $\ip{G_*}{\cdot}/(2\Sigma_*)$ is precisely what ensures that the SAV contribution to the energy change telescopes with the discrete gradient contribution, producing a single quadratic form $\ip{\eta}{S_*\,\eta}$ whose sign is determined by the structure matrix alone, as the next theorem makes precise.

\begin{theorem}[Modified Energy Evolution]\label{thm:modified-energy-evolution}
The scheme~\eqref{eq:general-dgsav} satisfies
\begin{equation}\label{eq:energy-evolution}
  \Ehat(x^{n+1},r^{n+1}) - \Ehat(x^n,r^n) = h\,\ip{\eta}{S_*\,\eta},
\end{equation}
where $\eta := \widetilde{\nabla}\Edg(x^{n+1},x^n) + \frac{\bar{r}}{\Sigma_*}\,G_*$.
\end{theorem}

\begin{proof}
Decompose $\Ehat^{n+1} - \Ehat^n = [\Edg(x^{n+1}) - \Edg(x^n)] + [(r^{n+1})^2 - (r^n)^2]$. The discrete chain rule yields the first bracket as $\ip{\widetilde{\nabla}\Edg}{x^{n+1} - x^n}$, while $(r^{n+1})^2 - (r^n)^2 = 2\bar r(r^{n+1} - r^n)$ together with the second equation of~\eqref{eq:general-dgsav} yields the second bracket as $(\bar r/\Sigma_*)\ip{G_*}{x^{n+1} - x^n}$. Summing gives $\ip{\eta}{x^{n+1} - x^n}$, and substituting the first equation of~\eqref{eq:general-dgsav} gives $h\,\ip{\eta}{S_*\,\eta}$.
\end{proof}

{The energy evolution~\eqref{eq:energy-evolution} has the same quadratic-form structure as the energy law template (Lemma~\ref{lem:energy-law}), with $\eta$ playing the role of the gradient after the SAV increment has been eliminated.}

\begin{corollary}[Conservation vs.\ Dissipation]\label{cor:conservation-dissipation}
The sign of the quadratic form in~\eqref{eq:energy-evolution} determines the qualitative behavior of the modified energy:
\begin{enumerate}[label=(\roman*)]
  \item \textbf{Conservation:} If $S_*^\top = -S_*$, then $\Ehat(x^{n+1},r^{n+1}) = \Ehat(x^n,r^n)$ for all~$n$.
  \item \textbf{Dissipation:} If $S_* = -M_*$ with $M_* \succeq 0$, then $\Ehat(x^{n+1},r^{n+1}) - \Ehat(x^n,r^n) = -h\,\ip{\eta}{M_*\,\eta} \le 0$ for all~$n$.
\end{enumerate}
\end{corollary}

\begin{proof}
Both cases follow from Theorem~\ref{thm:modified-energy-evolution} and the same quadratic-form argument as in Lemma~\ref{lem:energy-law}: $\ip{\eta}{S_*\,\eta} = 0$ when $S_*$ is skew-symmetric, and $\ip{\eta}{(-M_*)\,\eta} \le 0$ when $M_* \succeq 0$.
\end{proof}

\section{DG--SAV Time Integrators: Variants and Analysis}
\label{sec:variants}
\label{sec:dissipative}
\label{sec:hamiltonian}

Scheme~\eqref{eq:general-dgsav} is parameterized by a structure matrix $S_*$ and an evaluation point $x_*$ for the SAV coefficients $(\Sigma_*, G_*)$. This section instantiates it to obtain three concrete variants---Forward Euler, Midpoint, and Predictive---and analyzes their energy laws, accuracy, and well-posedness. The dissipative case ($S(x) = -M(x)$, $M\succeq 0$) and the conservative/Hamiltonian case ($S(x) = J(x)$, $J^\top = -J$) are treated together with $E = H$ in the latter; the splitting $E = \Edg + \Esav$ and the modified energy $\Ehat(x,r) = \Edg(x) + r^2$ are common to both.


\subsection{The Three Variants}
\label{sec:three-variants}

{
All three variants have the form~\eqref{eq:general-dgsav}. They differ only in the reference state $x_*$ used for the SAV coefficients and, when $S$ is state dependent, for $S_*$.
}

\begin{definition}[The three DG--SAV variants]\label{def:three-variants}
The three variants of the unified template~\eqref{eq:general-dgsav} are defined by:
\begin{enumerate}[label=(\roman*)]
  \item {\textbf{Forward Euler} (\textit{first order}): set $x_*=x^n$. Then $S_*=S(x^n)$, $\Sigma_*=\sqrt{\Esav(x^n)+C_0}$, and $G_*=\nabla\Esav(x^n)$.}
  \item {\textbf{Midpoint} (\textit{second order, self-adjoint}): set $x_*=\bar x:=\tfrac12(x^{n+1}+x^n)$. Then $S_*=S(\bar x)$, $\Sigma_*=\sqrt{\Esav(\bar x)+C_0}$, and $G_*=\nabla\Esav(\bar x)$.}
  \item {\textbf{Predictive} (\textit{second order, reduced cost}): set $x_*=x_{n+1/2}^{\mathrm{pred}}$, where the predictor is one explicit half-step from $(x^n,r^n)$:}
        \begin{equation}\label{eq:predictive-dgsav-pred}
          x_{n+1/2}^{\mathrm{pred}} = x^n + \tfrac{h}{2}\,S(x^n)\!\left(\grad\Edg(x^n) + \tfrac{r^n}{\sqrt{\Esav(x^n)+C_0}}\,\grad\Esav(x^n)\right),
        \end{equation}
        and $S_* = S(x_{n+1/2}^{\mathrm{pred}})$, $\Sigma_* = \sqrt{\Esav(x_{n+1/2}^{\mathrm{pred}})+C_0}$, $G_* = \grad\Esav(x_{n+1/2}^{\mathrm{pred}})$.
\end{enumerate}
With the chosen $(S_*, \Sigma_*, G_*)$, the variant scheme is
\begin{equation}\label{eq:euler-dgsav}\tag{DG--SAV variant}
  \begin{aligned}
    \frac{x^{n+1}-x^n}{h} &= S_*\!\left(\widetilde{\nabla}\Edg(x^{n+1},x^n) + \frac{\bar r}{\Sigma_*}\,G_*\right), \\[4pt]
    r^{n+1}-r^n &= \frac{\ip{G_*}{x^{n+1}-x^n}}{2\Sigma_*}, \qquad \bar r := \tfrac{1}{2}(r^{n+1}+r^n).
  \end{aligned}
\end{equation}
\end{definition}

The Forward Euler variant is mildly implicit (only the discrete gradient $\widetilde{\nabla}\Edg$ couples to $x^{n+1}$); the Midpoint variant is fully implicit; the Predictive variant restores mild implicitness while retaining second-order accuracy.

\subsection{Energy laws and accuracy}
\label{sec:variants-energy-accuracy}

\begin{proposition}[Modified energy law for the three variants]\label{prop:variants-energy-law}
For each variant in Definition~\ref{def:three-variants}, the modified energy satisfies
\[
  \Ehat(x^{n+1},r^{n+1}) - \Ehat(x^n,r^n) = h\,\ip{\eta}{S_*\,\eta}, \qquad \eta := \widetilde{\nabla}\Edg(x^{n+1},x^n) + \frac{\bar r}{\Sigma_*}\,G_*.
\]
{By Corollary~\ref{cor:conservation-dissipation}, $S(x) = -M(x)$ with $M\succeq 0$ yields $\Ehat^{n+1} - \Ehat^n = -h\ip{\eta}{M_*\,\eta} \le 0$ (dissipation), while $S(x)^\top = -S(x)$ yields $\Ehat^{n+1} = \Ehat^n$ (conservation). The identity holds without a time-step restriction for every solvable step.}
\end{proposition}

\begin{proof}
Each variant is an instance of~\eqref{eq:general-dgsav}; the result is then Theorem~\ref{thm:modified-energy-evolution} together with Corollary~\ref{cor:conservation-dissipation}.
\end{proof}

\begin{theorem}[Accuracy of the three variants]\label{thm:variants-accuracy}\label{thm:predictive-diss-properties}
{Assume $\Edg,\Esav \in C^3$, the structure map $S(\cdot)$ is $C^2$, and $\widetilde{\nabla}\Edg$ is a symmetric discrete gradient satisfying the hypotheses of Lemma~\ref{lem:symmDG-midpoint}. In addition, assume that along the exact trajectory and in a sufficiently small neighborhood of it,
\[
  \Esav(x)+C_0 \ge \delta > 0,
\]
so that $\Sigma$, $G/\Sigma$, and the ratios involving $\bar r/\Sigma_*$ are smooth and uniformly bounded. Then:}
\begin{enumerate}[label=(\roman*)]
  \item Forward Euler is first-order accurate ($O(h)$ global error);
  \item Midpoint is self-adjoint ($\Phi_{-h} = \Phi_h^{-1}$) and second-order accurate;
  \item Predictive is second-order accurate.
\end{enumerate}
\end{theorem}

\begin{proof}
{
For Forward Euler, the SAV coefficients are evaluated at $x^n$; direct consistency therefore gives a one-step defect of $O(h^2)$ and first-order global accuracy.

For Midpoint, interchanging $(x^n,r^n)$ and $(x^{n+1},r^{n+1})$ while replacing $h$ by $-h$ leaves every midpoint quantity and the symmetric discrete gradient unchanged. Since the method is consistent and self-adjoint, the standard result that a symmetric one-step method has even order~\cite[Theorem~II.3.2, pp.~42--43]{HairerLubichWanner2006} implies that it is at least second-order accurate.

For Predictive, start from exact data $x^n=x(t_n)$ and $r^n=r(t_n)$. The explicit half-step satisfies $x_{n+1/2}^{\mathrm{pred}}=x(t_{n+1/2})+O(h^2)$, so smoothness makes $S_{\mathrm{pred}}$, $\Sigma_{\mathrm{pred}}$, and $G_{\mathrm{pred}}$ midpoint-accurate to $O(h^2)$. Moreover, \cref{lem:symmDG-midpoint} gives
\[
\widetilde{\nabla}\Edg(x(t_{n+1}),x(t_n))
=\nabla\Edg(x(t_{n+1/2}))+O(h^2),
\]
and the exact constraint $r=\Sigma(x)$ implies $\bar r/\Sigma_{\mathrm{pred}}=1+O(h^2)$. Hence the generalized gradient, $\eta$, in the $x$-equation equals $\nabla E(x(t_{n+1/2}))+O(h^2)$, producing an $O(h^3)$ residual. Substitution into the $r$-equation and the identity $2r\dot r=\ip{G}{\dot x}$ give the same $O(h^3)$ residual for $r$. Local well-posedness then converts these residual estimates into an $O(h^3)$ one-step defect and $O(h^2)$ global error. The argument uses no sign property of $S$, so it also applies to the conservative case $S=J$. A complete residual calculation is given in Appendix~\ref{sec:supp-proof-theorem-44}.
}
\end{proof}

\subsection{Well-posedness per step}
\label{sec:variants-wellposed}

\begin{proposition}[Local existence and uniqueness for frozen-coefficient variants]\label{prop:well-posedness-diss}
For the Forward Euler and Predictive variants, $(S_*, \Sigma_*, G_*)$ depend only on known data. Fix $R > 0$ and let $B = \{x : \|x - x^n\| \le R\}$. Under the assumptions
\begin{enumerate}[label=(\alph*)]
  \item $\Edg, \Esav \in C^2$ with bounded second derivatives on bounded sets;
  \item $\|S_*\| \le \mu_S < \infty$;
  \item $\widetilde{\nabla}\Edg(\cdot, x^n)$ is Lipschitz on $B$ with constant $L_{\mathrm{DG}}(R)$;
  \item $C_0$ is chosen so that $\Sigma_*^2 = \Esav(x_*) + C_0 \ge \delta > 0$,
\end{enumerate}
there exists $h_0>0$, depending on these bounds and on the known current data, such that for all $0<h<h_0$ the implicit equations~\eqref{eq:general-dgsav} admit a unique solution $(x^{n+1}, r^{n+1})$ with $x^{n+1} \in B$.
\end{proposition}

\begin{proof}
Eliminate $r^{n+1}$ from the second equation of~\eqref{eq:general-dgsav}, $r^{n+1} = r^n + \ip{G_*}{x^{n+1}-x^n}/(2\Sigma_*)$, so that $\bar r$ is an affine function of $x^{n+1}$ with Lipschitz constant $\|G_*\|/(4\Sigma_*)$. Substituting into the first equation recasts the system as a fixed-point problem $x^{n+1} = T(x^{n+1})$ on $B$, where
\[
  T(x) := x^n + h\,S_*\!\left(\widetilde{\nabla}\Edg(x, x^n) + \tfrac{\bar r(x)}{\Sigma_*}G_*\right).
\]
For $x, y \in B$, hypothesis~(c) and the affine $\bar r$-dependence give $\|T(x) - T(y)\| \le h\,\mu_S\bigl(L_{\mathrm{DG}}(R) + \|G_*\|^2/(4\Sigma_*^2)\bigr)\|x - y\|$, with $\|G_*\|^2/(4\Sigma_*^2) \le \|G_*\|^2/(4\delta)$ by hypothesis~(d). Together with the invariance bound $\|T(x) - x^n\| \le h\,K(R)$ (triangle inequality from the same estimate plus $\|T(x^n)-x^n\| = h\,\rho_0$ with $\rho_0 := \|S_*(\widetilde{\nabla}\Edg(x^n,x^n) + (r^n/\Sigma_*)G_*)\|$), the threshold $h < h_0$ makes $T : B \to B$ a contraction on $(B, \|\cdot\|)$. Banach's fixed-point theorem yields the unique $x^{n+1} \in B$, and $r^{n+1}$ is determined by the second equation. Both the dissipative case $S_* = -M_*$ and the conservative case $S_* = J_*$ are covered.
\end{proof}

\begin{remark}[Well-posedness of the midpoint variant]
For the midpoint variant the coefficients $(S(\bar x), \Sigma_{\bar x}, G_{\bar x})$ depend on the unknown through $\bar x$, so the Banach argument above does not directly apply. Under the smoothness and positivity assumptions of \cref{thm:variants-accuracy}, the implicit function theorem applied to the midpoint system at $(x^{n+1},r^{n+1},h)=(x^n,r^n,0)$ yields a unique smooth local branch of solutions for sufficiently small~$h$.
\end{remark}

\section{Extension to Poisson Systems}
\label{sec:poisson}

{The augmented skew-tensor formulation of Section~\ref{sec:continuous-tensor-form} provides a natural starting point for extension to Poisson systems. There, the augmented skew-symmetric tensor $K(x)$ together with the discrete-gradient construction was sufficient for conservation of the modified Hamiltonian $\Hhat$. In the Poisson setting, the continuous structure tensor additionally satisfies the Jacobi identity and admits Casimir invariants---functions that are conserved for \emph{any} choice of Hamiltonian. The discrete constructions below use consistent skew-symmetric two-point tensors enforcing selected discrete Casimir conditions; they are not claimed to satisfy a discrete Jacobi identity.} Poisson integrators exploiting a dimension-extended reformulation---which is itself Poisson---were recently derived by Lu, Wang, and Sun~\cite{LuWangSun2025}; the construction below is complementary, enforcing selected discrete Casimir conditions for the original tensor within the frozen-coefficient DG--SAV template.

\subsection{Continuous Poisson form, Casimirs, and the discrete Casimir condition}
\label{sec:casimir-projection}

Consider Poisson dynamics $\dot z = B(z)\,\nabla H(z)$ with skew-symmetric Poisson tensor $B(z)^\top = -B(z)$, and let a \emph{Casimir} be a function $C:\R^d \to \R$ satisfying $B(z)\,\nabla C(z) = 0$ for all~$z$; Casimirs are conserved along all Hamiltonian flows with Poisson tensor $B$, regardless of $H$. For discretization we use a consistent skew-symmetric two-point tensor $\widetilde{B}(z^+, z^-)$ (with $\widetilde{B}(z,z) = B(z)$), the most common choice being the midpoint evaluation $\widetilde{B}(z^+, z^-) = B(\bar z)$ with $\bar z = \tfrac12(z^++z^-)$. For the discrete scheme to preserve Casimirs, $\widetilde{B}$ must satisfy the \emph{discrete Casimir condition}
\begin{equation}\label{eq:discrete-casimir}
  \widetilde{B}(z^+, z^-)\,\widetilde{\nabla} C(z^+, z^-) = 0.
\end{equation}
The midpoint evaluation does \emph{not} satisfy~\eqref{eq:discrete-casimir} in general; the projected tensor construction below enforces it.

\begin{theorem}[Casimir-Compatible Projected Tensor]\label{thm:projected-B}
Let $C_1,\ldots,C_k$ be Casimirs with discrete gradients $\widetilde{\nabla} C_j$. Assume:
\begin{enumerate}[label=(\alph*)]
  \item The discrete gradients $\{\widetilde{\nabla} C_1(z^+,z^-), \ldots, \widetilde{\nabla} C_k(z^+,z^-)\}$ are linearly independent for all $(z^+, z^-)$ in the region of interest.
  \item {The selected functions are Casimirs of the continuous tensor, $B(z)\nabla C_j(z)=0$ for all $j$ and all $z$ in the region of interest.}
\end{enumerate}
For any pair $(z^+,z^-)$, define the matrix of discrete Casimir gradients
\[
  \mathcal{C}(z^+,z^-) = \big[\widetilde{\nabla} C_1(z^+,z^-) \;\cdots\; \widetilde{\nabla} C_k(z^+,z^-)\big] \in \R^{d \times k},
\]
the orthogonal projection
\[
  P(z^+,z^-) = I - \mathcal{C}(\mathcal{C}^\top \mathcal{C})^{-1}\mathcal{C}^\top,
\]
onto $(\mathrm{col}\,\mathcal{C}(z^+,z^-))^\perp$ (where $\mathrm{col}\,\mathcal{C}$ denotes the column space, i.e.\ the set of all linear combinations of the columns of $\mathcal{C}(z^+,z^-)$), and the corrected tensor
\[
  \widetilde{B}_P(z^+,z^-) = P(z^+,z^-)\,\widetilde{B}_0(z^+,z^-)\,P(z^+,z^-),
\]
where $\widetilde{B}_0(z^+,z^-)$ is any consistent skew-symmetric approximation of $B$.

{Then $\widetilde{B}_P$ is a consistent skew-symmetric two-point tensor satisfying the selected discrete Casimir conditions $\widetilde{B}_P\,\widetilde{\nabla} C_j = 0$ for all $j$. No discrete Jacobi identity is asserted.}
\end{theorem}
In practice, the principal difficulty in applying Theorem~\ref{thm:projected-B} is the explicit construction of discrete gradients for the Hamiltonian and for each Casimir $C_j$.

\begin{proof}
\emph{Skew-symmetry.} Assumption~(a) makes $\mathcal{C}^\top \mathcal{C}$ invertible, so $P$ is well-defined, and $P^\top = P$ gives $\widetilde{B}_P^\top = (P\,\widetilde{B}_0\,P)^\top = -P\,\widetilde{B}_0\,P = -\widetilde{B}_P$.

{\emph{Consistency.} At $z^+ = z^- = z$, $\mathcal{C}(z,z) = [\nabla C_1(z)\;\cdots\;\nabla C_k(z)]$, so $P(z,z)$ projects onto $(\mathrm{span}\{\nabla C_j(z)\})^\perp$. Since $B(z)\nabla C_j(z)=0$ and $B(z)^\top=-B(z)$, the range of $B(z)$ is orthogonal to each selected $\nabla C_j(z)$, hence $P(z,z)B(z)=B(z)$. Also $B(z)P(z,z)=B(z)$ because $B(z)$ annihilates the selected Casimir-gradient directions removed by $P$. Therefore $P(z,z)B(z)P(z,z)=B(z)$, and consistency of $\widetilde{B}_0$ gives $\widetilde{B}_P(z,z)=B(z)$.}

\emph{Discrete Casimir condition.} By definition of $P$, $P\,\widetilde{\nabla}C_j = 0$ for each $j$ (since $\widetilde{\nabla}C_j$ is a column of $\mathcal{C}$), so $\widetilde{B}_P\,\widetilde{\nabla}C_j = P\,\widetilde{B}_0\,P\,\widetilde{\nabla}C_j = 0$.
\end{proof}

\subsection{Midpoint DG--SAV Poisson Scheme}

\begin{definition}[Midpoint DG--SAV for Poisson Systems]
With $\bar{z} = \frac{1}{2}(z^{n+1} + z^n)$ and Poisson tensor $B$:
\begin{equation}\label{eq:poisson-midpoint}
  \begin{aligned}
    z^{n+1} - z^n &= h\,B(\bar{z})\left(\widetilde{\nabla}\Hdg(z^{n+1},z^n) + \frac{\bar{r}}{\sqrt{\Hsav(\bar{z})+C_0}}\,\nabla\Hsav(\bar{z})\right), \\[6pt]
    r^{n+1} - r^n &= \frac{\ip{\nabla\Hsav(\bar{z})}{z^{n+1} - z^n}}{2\sqrt{\Hsav(\bar{z})+C_0}}.
  \end{aligned}
\end{equation}
The \emph{projected} variant used in our Casimir experiments (\cref{sec:exp-projected-casimir}) is obtained by replacing $B(\bar z)$ in the first equation by the projected tensor $\widetilde{B}_P(z^{n+1},z^n)$ of \cref{thm:projected-B}; both equations remain coupled through $z^{n+1}, z^n$.
\end{definition}

\begin{theorem}[Properties]\label{thm:poisson-midpoint}
{
Assume $\Hdg,\Hsav \in C^3$, and that $\widetilde{\nabla}\Hdg$ is a symmetric discrete gradient satisfying the hypotheses of \cref{lem:symmDG-midpoint}. Let $\widetilde{B}(z^{n+1},z^n)$ be a $C^2$ two-point tensor in a neighborhood of the diagonal that is consistent, skew-symmetric, and invariant under interchange of its two arguments. This includes $\widetilde{B}=B(\bar z)$ for a $C^2$ skew-symmetric Poisson tensor and the projected tensor $\widetilde{B}_P$ of \cref{thm:projected-B} when $\widetilde{B}_0$ and the Casimir discrete gradients are $C^2$ in a neighborhood of the diagonal and invariant under the same interchange. Then the scheme~\eqref{eq:poisson-midpoint}, with $B(\bar z)$ replaced by $\widetilde{B}$:
}
\begin{enumerate}[label=(\roman*)]
  \item conserves the modified Hamiltonian exactly;
  \item is self-adjoint and second-order accurate;
  \item preserves Casimirs if the discrete Casimir condition $\widetilde{B}\,\widetilde{\nabla}C = 0$ holds.
\end{enumerate}
\end{theorem}

\begin{proof}
\emph{(i) Modified Hamiltonian conservation.} {Write the method in the augmented tensor form of \cref{sec:continuous-tensor-form}, with $K_h$ assembled from $\widetilde{B}$ and the SAV coefficients. The modified-Hamiltonian increment is $h\,\widetilde{\nabla}\Hhat^\top K_h\widetilde{\nabla}\Hhat$, which vanishes because $K_h$ inherits skew-symmetry from $\widetilde{B}$.}

\emph{(ii) Self-adjointness and second-order accuracy.} {Interchanging $(z^n,r^n)$ and $(z^{n+1},r^{n+1})$ and replacing $h$ by $-h$ leaves $\bar z$, $\bar r$, $\widetilde{\nabla}\Hdg$, and $\widetilde{B}$ unchanged, while reversing both increments. Thus the scheme is self-adjoint. The $O(h^2)$ midpoint consistency of \cref{lem:symmDG-midpoint}, smoothness of $\widetilde{B}$, $\Hsav$, and $\nabla\Hsav$, and the order-doubling theorem~\cite{HairerLubichWanner2006} give second-order accuracy.}

\emph{(iii) Casimir preservation.} With $\eta = \widetilde{\nabla}\Hdg + \bigl(\bar r/\sqrt{\Hsav(\bar z)+C_0}\bigr)\,\nabla\Hsav(\bar z)$ and the discrete Casimir condition $\widetilde{B}\,\widetilde{\nabla}C = 0$,
\[
  C(z^{n+1}) - C(z^n) = \ip{\widetilde{\nabla}C}{z^{n+1} - z^n} = h\,\ip{\widetilde{\nabla}C}{\widetilde{B}\,\eta} = -h\,\eta^\top \widetilde{B}\,\widetilde{\nabla}C = 0
\]
by skew-symmetry of $\widetilde{B}$ and the discrete Casimir condition.
\end{proof}

\section{Separable Hamiltonians and Explicit DG--SAV Schemes}
\label{sec:separable}

{
For separable Hamiltonians $H(q,p)=T(p)+V(q)$ with quadratic kinetic energy, the DG--SAV framework yields two fully explicit schemes: the Forward Euler and Predictive variants. With the paper-wide normalization $r^2=V+C_0$, both exactly conserve the modified Hamiltonian by \cref{cor:conservation-dissipation}. This specialization uses the same scalar quadratization as Bilbao, Ducceschi, and Zama~\cite{Bilbao2023}, with the normalizations related by $\psi_{\mathrm{BDZ}}=\sqrt{2}\,r$ and $g_{\mathrm{BDZ}}=g/\sqrt{2}$. The time discretizations are presented differently: Bilbao, Ducceschi, and Zama employ centered staggered variables, whereas the schemes below are collocated one-step specializations of the general DG--SAV template.
}

\subsection{Setup}
\label{sec:separable-setup}

{
Take the canonical Hamiltonian system $\dot z=J\nabla H(z)$, $z=(q,p)$, $J=\bigl(\begin{smallmatrix}0&I\\-I&0\end{smallmatrix}\bigr)$, with $H(q,p)=\tfrac12p^\top M^{-1}p+V(q)$. We split
}
\[
  \Hdg(q,p) = \tfrac12 p^\top M^{-1}p, \qquad \Hsav(q) = V(q),
\]
{
so the DG part is quadratic and the nonlinear dependence is confined to the SAV part. Choose $C_0$ so that $V(q)+C_0\ge\delta>0$ on the relevant region; if a global lower bound is used, take $C_0>-\inf_qV(q)$. Define
}
\[
  r(q):=\sqrt{V(q)+C_0}, \qquad
  g(q):=\frac{\nabla V(q)}{\sqrt{V(q)+C_0}}, \qquad
  \nabla V(q)=r(q)\,g(q).
\]
{
This is the same normalization used in the general framework. The modified Hamiltonian is
}
\[
  \Hhat(q,p,r)=\tfrac12p^\top M^{-1}p+r^2=H(q,p)+C_0.
\]
{
Since $\Hdg$ is quadratic, its discrete gradient is $\widetilde{\nabla}\Hdg=(0,M^{-1}\bar p)^\top$, where $\bar p:=\tfrac12(p_{n+1}+p_n)$.
}

\subsection{Triangular elimination}
\label{sec:separable-lemma}

The shared algebra of the Forward Euler and Predictive variants is summarized in the following lemma.

\begin{lemma}[Explicit triangular resolution]\label{lem:triangular-resolution}
Fix $g_*\in\R^d$ and $\bar r:=\tfrac12(r_{n+1}+r_n)$. The corrector system
\begin{equation}\label{eq:sep-corrector}
\begin{aligned}
  q_{n+1} &= q_n + \tfrac h2 M^{-1}(p_{n+1}+p_n), \quad
  p_{n+1} = p_n - \tfrac h2 g_*(r_{n+1}+r_n), \\
  r_{n+1} - r_n &= \tfrac h4 g_*^\top M^{-1}(p_{n+1}+p_n)
\end{aligned}
\end{equation}
admits the explicit triangular solution
\begin{equation}\label{eq:sep-explicit}
  r_{n+1} = \frac{(1-\alpha_*)\,r_n + \tfrac h2\,g_*^\top M^{-1}p_n}{1+\alpha_*}, \qquad
  \alpha_* := \tfrac{h^2}{8}\,g_*^\top M^{-1} g_*,
\end{equation}
{
followed by $p_{n+1}=p_n-\tfrac h2g_*(r_{n+1}+r_n)$ and $q_{n+1}=q_n+\tfrac h2M^{-1}(p_{n+1}+p_n)$. Thus the update $r_{n+1}\to p_{n+1}\to q_{n+1}$ requires no nonlinear solve.
}
\end{lemma}

\begin{proof}
{
Substitution of the $p_{n+1}$ equation into the $r$-equation gives
\[
r_{n+1}-r_n=\tfrac h2g_*^\top M^{-1}p_n-\alpha_*(r_{n+1}+r_n).
\]
Rearranging yields \eqref{eq:sep-explicit}; the $p_{n+1}$ and $q_{n+1}$ updates then follow directly.
}
\end{proof}

{
The corrector~\eqref{eq:sep-corrector} is the unified template~\eqref{eq:euler-dgsav} specialized to $S=J$, $\Hdg=T$, and $\Hsav=V$, with $g_*=\nabla V/\sqrt{V+C_0}$ evaluated at a state-dependent reference point. The two explicit variants correspond to the two reference states in \cref{def:three-variants}.
}

\subsection{Euler and Predictive variants as corollaries}
\label{sec:separable-variants}
\label{sec:separable-predictive}

\begin{corollary}[Fully explicit Euler and Predictive variants]\label{cor:separable-explicit}\label{cor:separable-euler}\label{cor:separable-predictive}
Specializing \cref{def:three-variants} to the separable case in \cref{lem:triangular-resolution} yields:
\begin{enumerate}[label=(\roman*)]
  \item \emph{Forward Euler:} $g_* = g(q_n)$. The scheme is fully explicit, exactly preserves $\Hhat$ by \cref{cor:conservation-dissipation}(i), and is first-order accurate by \cref{thm:variants-accuracy}\textnormal{(i)}; the loss of order arises from freezing $g$ at the left endpoint.
  \item \emph{Predictive:} the unified predictor~\eqref{eq:predictive-dgsav-pred} reduces in the separable case to the explicit Euler half-step $q_{n+1/2}^{\mathrm{pred}} = q_n + \tfrac h2 M^{-1} p_n$, and we set $g_* = g\bigl(q_{n+1/2}^{\mathrm{pred}}\bigr)$. The resulting scheme is fully explicit, exactly preserves $\Hhat$, and is second-order accurate by \cref{thm:variants-accuracy}\textnormal{(iii)}: the $O(h^2)$ predictor accuracy propagates to $g_*$ by smoothness, and the quadratic kinetic part contributes no further error.
\end{enumerate}
\end{corollary}

\section{Numerical Experiments}
\label{sec:numerics}

{
Four experiments isolate the settings addressed by the framework. A logarithmic Flory--Huggins Allen--Cahn equation compares the DG--SAV variants with classical SAV and Full DG for a non-polynomial dissipative PDE. A non-quadratic Ohta--Kawasaki-type model tests whether the energy split can retain Full DG accuracy while reducing the cost of a nonlocal discrete gradient. A separable double-well Hamiltonian illustrates the explicit schemes of \cref{sec:separable}. Finally, a Poisson system with a nonlinear cubic Casimir tests the projected tensor of \cref{sec:casimir-projection}.
}

\paragraph{{Timing protocol}}
{
All reported timings are wall-clock means over repeated runs on the same machine after one warm-up run. Per-step times exclude one-time grid and operator assembly but include linear-solver work and, where applicable, all Picard or Newton iterations. The reported ``solves/step'' counts include every inner iteration required to reach the stated tolerance.
}

\subsection{Dissipative PDE: Allen--Cahn with logarithmic Flory--Huggins potential}
\label{sec:num-allen-cahn}

\paragraph{{Problem}}
We consider the Allen--Cahn equation $u_t = -\delta E/\delta u$ with the logarithmic Flory--Huggins free energy
\begin{equation}\label{eq:ac-fh-energy}
  E(u) = \tfrac{\varepsilon^2}{2}\!\int_\Omega |\nabla u|^2\,dx
       + \theta\!\int_\Omega \bigl[u\ln u + (1-u)\ln(1-u)\bigr] dx
       - \tfrac{\theta_c}{2}\!\int_\Omega \bigl(u-\tfrac12\bigr)^2 dx,
\end{equation}
{
This is a standard model of binary phase separation in polymer mixtures and binary alloys~\cite{ShenXuYang2018}.
}

\paragraph{{DG--SAV splitting}}
We take
\begin{equation*}
\begin{aligned}
  \Edg(u) &= \tfrac{\varepsilon^2}{2}\!\int_\Omega |\nabla u|^2\,dx \quad\text{(quadratic, Fourier-implicit)}, \\
  \Esav(u) &= \int_\Omega F(u)\,dx \quad\text{(non-polynomial, SAV)}.
\end{aligned}
\end{equation*}
{
Because $\Edg$ is quadratic, its discrete gradient is the midpoint average and the Dirichlet term receives a Crank--Nicolson treatment; the full nonlinearity $F$ is assigned to the auxiliary variable. This example therefore compares Full DG, which evaluates the logarithmic AVF at every iteration, with SAV-type treatments that freeze the nonlinear coefficient. The $p=6$ benchmark in \cref{sec:num-nonquadratic-benchmark} considers the complementary case in which a non-quadratic term remains in $\Edg$. The logarithmic AVF is available in closed form through $\int\ln x\,dx=x\ln x-x$:
}
\begin{multline*}
  \widetilde{\nabla}_{\mathrm{AVF}} \!\!\int F(u)\,dx \;=\; \theta\,\frac{u^+\!\ln u^+ - u^-\!\ln u^-}{u^+ - u^-} \\
  + \theta\,\frac{(1-u^+)\!\ln(1-u^+) - (1-u^-)\!\ln(1-u^-)}{u^+ - u^-} - \theta_c\,\bigl(\tfrac{u^+ + u^-}{2} - \tfrac12\bigr),
\end{multline*}
{
Each Full DG Picard iteration nevertheless requires two pointwise logarithm evaluations per node, and the iteration converges more slowly near the steep regions close to $u=0$ and $u=1$ than for a polynomial double well. We use a Fourier spectral discretization with $N=128$ on $[0,2\pi)$, $\varepsilon=0.05$, $C_0=10$, and a fixed-seed random perturbation of amplitude $0.05$ around the unstable equilibrium $u=1/2$. The final time is $T=2$, except in the convergence study, where $T=0.5$.
}

\paragraph{{Convergence}}
{
\Cref{fig:fh-conv-gap-cost} (left) reports relative $L^2$ errors against a Full DG reference computed with $\Delta t=5\times10^{-4}$. Every second-order method attains the expected $O(\Delta t^2)$ rate; their principal differences are in the error constants.
}

\paragraph{{Cost--accuracy comparison}}
{
\Cref{fig:fh-conv-gap-cost,fig:fh-workprecision} show three cost levels. Full DG and Midpoint DG--SAV have the smallest errors and the highest costs, requiring about nine Picard iterations per step in this regime; Full DG and Midpoint DG--SAV have comparable per-step cost at essentially the same accuracy (their ordering is machine-dependent). SAV/CN and SAV/BDF2 have lower per-step costs but error constants that are $4$--$7\times$ larger. Predictive DG--SAV is second order and uses a non-iterative linear corrector with two nonlinear-gradient evaluations per step, one at the current state and one at the predicted half-step; in the tested moderate-accuracy range it is approximately four times more accurate than SAV/CN at comparable cost and roughly twice as fast as Full DG at comparable accuracy. Since $\Edg$ is quadratic here, the difference between Predictive DG--SAV and SAV/CN comes primarily from the explicit half-step predictor rather than the multistep extrapolant $\tfrac12(3u^n-u^{n-1})$.
}

\begin{figure}[!htbp]
\centering
\includegraphics[width=0.99\textwidth]{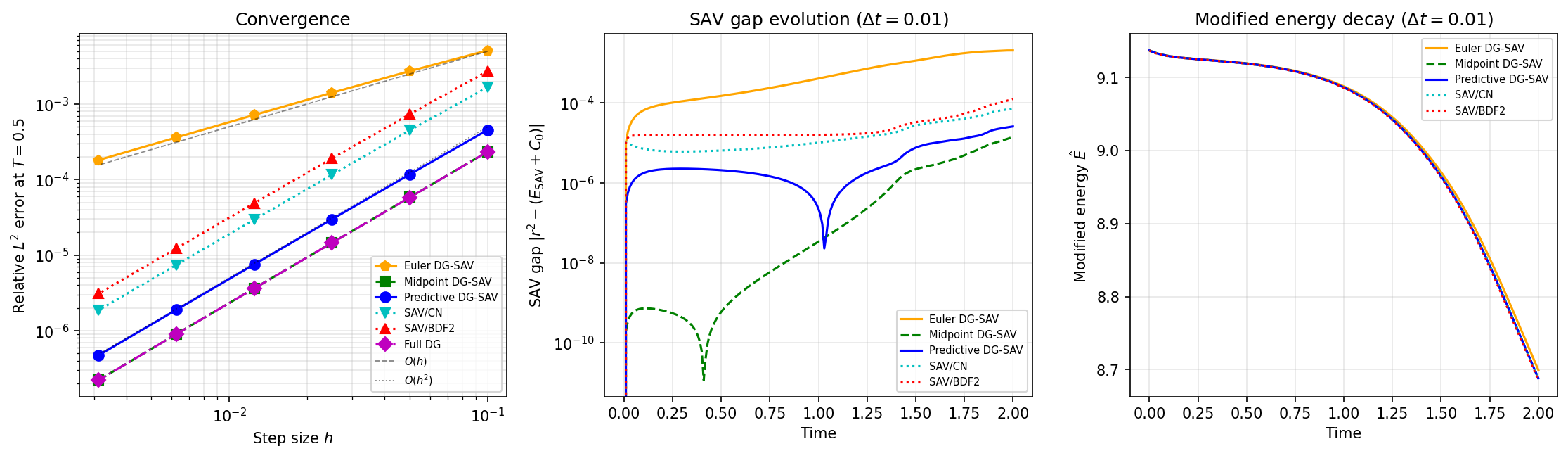}
\caption{Flory--Huggins Allen--Cahn benchmark ($N=128$, $\varepsilon=0.05$, $T=0.5$).
Left: $L^2$ convergence against the Full DG reference.
Center: SAV gap $|r^2-(\Esav+C_0)|$ at $\Delta t=0.01$, $T=2$; the second-order DG--SAV gaps remain below those of the classical SAV schemes, while Forward Euler DG--SAV has the largest gap.
Right: modified energy $\hat E$ at $\Delta t=0.01$, $T=2$; every SAV and DG--SAV method dissipates its modified energy monotonically. Full DG dissipates the original energy and is omitted. Wall-clock cost is reported in \cref{fig:fh-workprecision}.}
\label{fig:fh-conv-gap-cost}
\end{figure}

\begin{figure}[!htbp]
\centering
\includegraphics[width=0.7\textwidth]{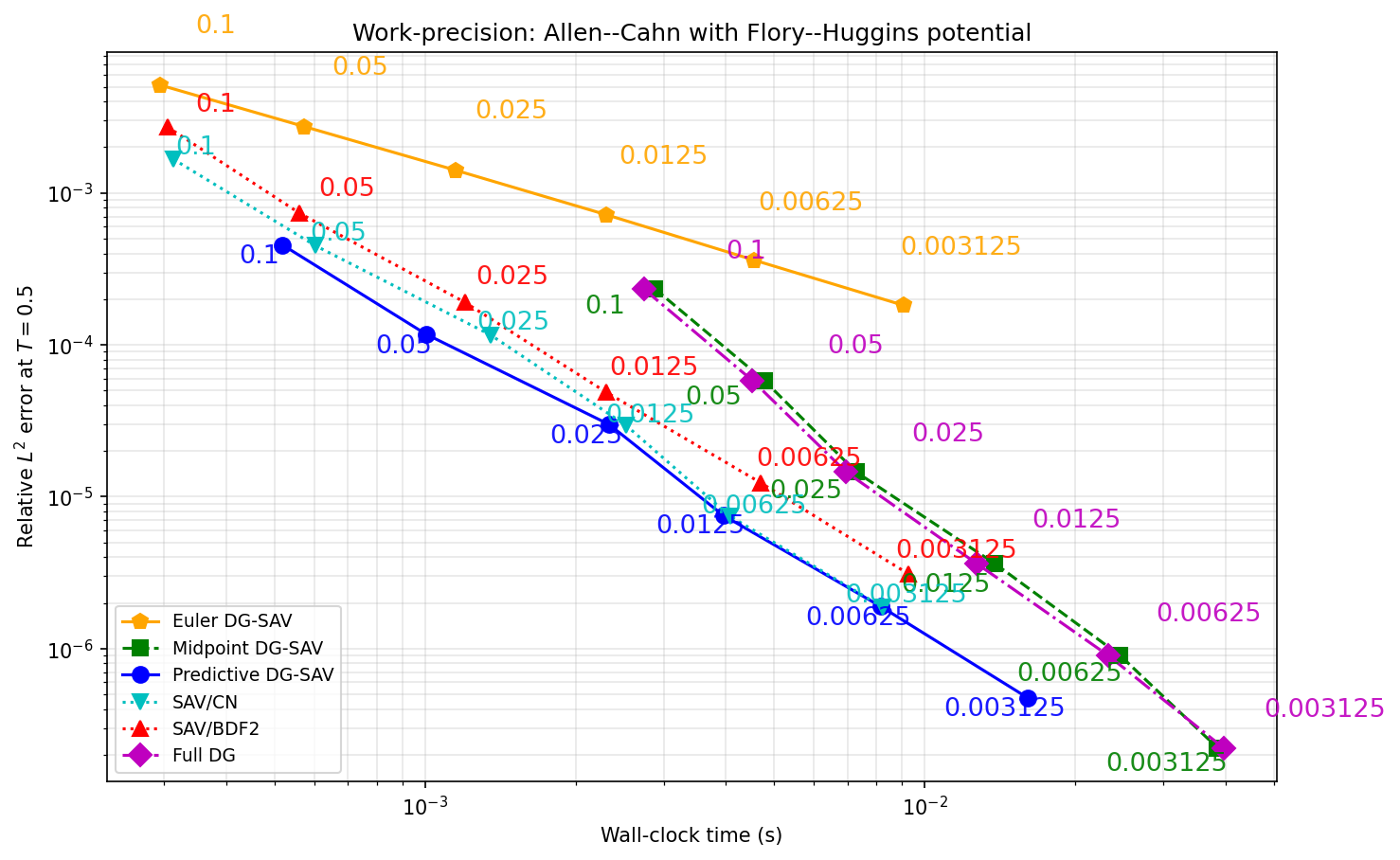}
\caption{Work--precision diagram for the Flory--Huggins benchmark. In the tested moderate-to-high-accuracy range, Predictive DG--SAV is more accurate than SAV/CN and SAV/BDF2 at comparable cost and cheaper than Full DG and Midpoint DG--SAV at comparable accuracy.}
\label{fig:fh-workprecision}
\end{figure}

\paragraph{{SAV gap}}
{
At $\Delta t=0.01$ and $T=2$, the final SAV gap is approximately $1\times10^{-5}$ for Midpoint DG--SAV and $3\times10^{-5}$ for Predictive DG--SAV, compared with $7\times10^{-5}$ for SAV/CN and $1\times10^{-4}$ for SAV/BDF2. Thus the second-order DG--SAV variants remain a few times closer to the constraint manifold in this experiment; Forward Euler DG--SAV has the larger gap $2\times10^{-3}$.
}

\paragraph{{Discussion}}
{
Appendix~\ref{sec:supp-polynomial-ac} reports the complementary polynomial double-well case, where the closed-form AVF makes Full DG competitive with Midpoint DG--SAV and leaves the non-iterative Predictive variant as the principal lower-cost alternative.
}

\subsection{Efficiency Benchmark: Non-Quadratic Nonlocal PDE}
\label{sec:num-nonquadratic-benchmark}

\paragraph{{Problem}}
We consider the $H^{-1}$ gradient flow of the Ohta--Kawasaki-type energy
\begin{equation}\label{eq:p6-energy}
  E[u] = \frac{\varepsilon^2}{2}\int_\Omega |\nabla u|^2\,dx
       + \frac{1}{4}\int_\Omega (u^2-1)^2\,dx
       + \frac{\gamma}{p}\int_\Omega \bigl|(-\Delta)^{-1/2}(u-m_0)\bigr|^p\,dx,
\end{equation}
{
Here $u_t=\Delta\mu$, $\mu=\delta E/\delta u$, on $\Omega=[0,2\pi)^2$ with periodic boundary conditions and prescribed mean $m_0$. The state space satisfies $\int_\Omega u\,dx=m_0|\Omega|$, and $(-\Delta)^{-1/2}$ acts on the mean-zero component $u-m_0$ through the periodic pseudoinverse. The case $p=2$ is the standard Ohta--Kawasaki diblock copolymer model; for $p\ne2$, the nonlocal term is non-quadratic. This experiment examines whether DG--SAV provides an intermediate accuracy--cost regime between classical SAV and Full DG.
}

\paragraph{{Why $p$ matters}}
{
For $p=2$, the midpoint discrete gradient of the nonlocal term is exact and can be included in the implicit Fourier operator. Full DG then dissipates the original energy exactly without an auxiliary variable, whereas assigning this quadratic term to SAV introduces a nonzero gap without simplifying the implicit operator. Appendix~\ref{sec:supp-polynomial-ac} documents the analogous closed-form-AVF regime for the polynomial double well. We therefore focus on $p=6$. The split assigns the Dirichlet and double-well terms
\[
\Edg(u)=\tfrac{\varepsilon^2}{2}\!\int_\Omega|\nabla u|^2\,dx
+\tfrac14\!\int_\Omega(u^2-1)^2\,dx
\]
to the discrete gradient and the nonlocal term
\[
\Esav(u)=\tfrac{\gamma}{6}\!\int_\Omega|(-\Delta)^{-1/2}(u-m_0)|^6\,dx
\]
to the auxiliary variable. Full DG evaluates the nonlocal AVF through a polynomial expansion requiring approximately six additional FFTs per Picard iteration; DG--SAV avoids those evaluations.

We compare SAV/CN and SAV/BDF2, which assign both nonlinear terms to one auxiliary variable; Forward Euler, Predictive, and Midpoint DG--SAV, which retain the pointwise double well in a closed-form AVF; and Full DG, which applies AVF discrete gradients to both nonlinear contributions and dissipates the original energy exactly.
}

\paragraph{{Numerical setup}}
{
We use a $64\times64$ Fourier grid, $\varepsilon=0.05$, $\gamma=5$, $m_0=0$, $C_0=1$, and $T=1$. The fixed-seed initial data are $u_0=0.1(\xi-\overline\xi)$ with $\xi\sim\mathcal N(0,1)$, and hence have exactly zero numerical mean. The reference solution is computed with SciPy's DOP853 method using \texttt{rtol}$=10^{-12}$ and \texttt{atol}$=10^{-14}$.
}

\paragraph{{Convergence}}
{
\Cref{fig:p6-convergence} (left) reports the $L^2$ errors. Predictive and Midpoint DG--SAV agree with Full DG to three significant digits across the tested step sizes and enter the asymptotic $O(\Delta t^2)$ regime near $\Delta t=10^{-3}$. At the finest step size, $\Delta t=10^{-5}$, the second-order error curves flatten near $8\times10^{-7}$; this point appears to be limited by the numerical reference and is therefore not used to infer the asymptotic order. SAV/CN and SAV/BDF2 converge with larger errors and a later asymptotic onset; at $\Delta t=2\times10^{-4}$, their errors remain one to two orders of magnitude larger than those of DG--SAV. These results indicate that one SAV scalar does not track the two nonlinear contributions as accurately as the split treatment in this test. Forward Euler DG--SAV is formally first order, but its first-order component is not resolved over the tested step-size range; its SAV gap nevertheless separates the frozen-coefficient approximation from the second-order variants.
}

\paragraph{{Energy decay and SAV gap}}
{
No increase in the relevant energy was observed: SAV/CN, SAV/BDF2, and the DG--SAV variants dissipate their modified energies, while Full DG dissipates the original energy. At $\Delta t=2\times10^{-4}$, the SAV gap is approximately $3\times10^{-10}$ for Midpoint DG--SAV and $5\times10^{-10}$ for Predictive DG--SAV, compared with $3\times10^{-3}$ for SAV/CN and $10^{-2}$ for SAV/BDF2. These absolute gaps correspond to different energy splits and therefore should not be interpreted as a scale-invariant comparison: the DG--SAV scalar tracks only the nonlocal term, whereas the classical SAV scalar tracks both nonlinear contributions.
}

\begin{figure}[!htbp]
\centering
\includegraphics[width=0.95\textwidth]{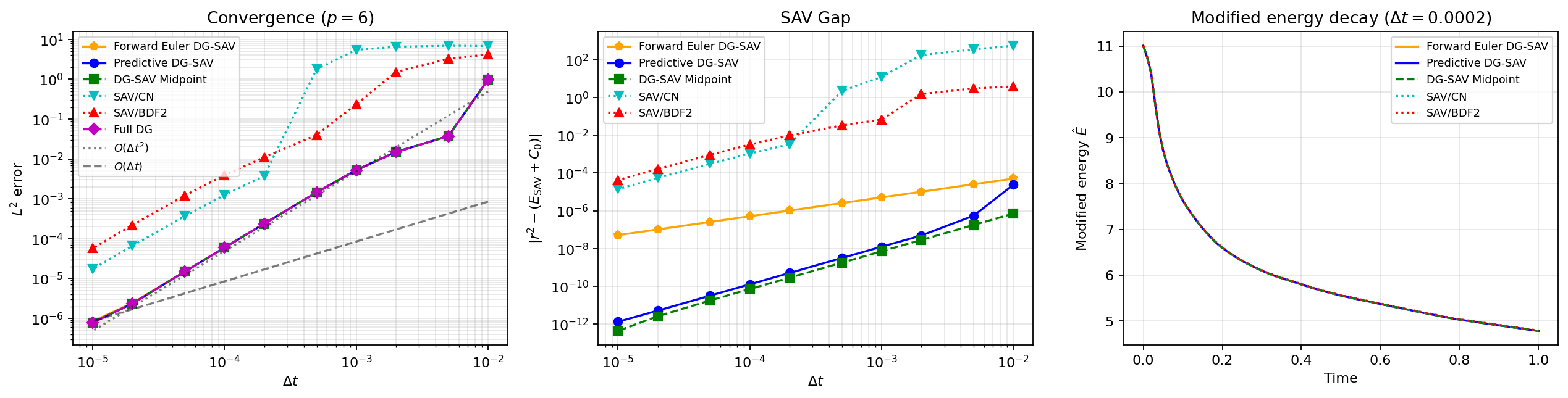}
\caption{$p=6$ nonlocal benchmark.
Left: $L^2$ convergence.
Center: SAV gap $|r^2-(\Esav+C_0)|$.
Right: modified energy $\hat E$ at $\Delta t=2\times10^{-4}$. The SAV and DG--SAV methods dissipate their modified energies; Full DG, which dissipates the original energy, is omitted.}
\label{fig:p6-convergence}
\end{figure}

\paragraph{{Efficiency and work--precision}}
{
\Cref{fig:p6-workprecision} shows that, at matched accuracy, Full DG is typically about twice as expensive as Predictive DG--SAV and about $1.2$--$1.5\times$ as expensive as Midpoint DG--SAV; both DG--SAV variants avoid evaluating the nonlocal AVF in every Picard iteration. The non-iterative SAV/CN and SAV/BDF2 steps are cheaper still, and their lower per-step cost can offset their larger error constants. SAV/BDF2 crosses Predictive DG--SAV near an $L^2$ error of $3.4\times10^{-3}$: BDF2 is cheaper at coarser tolerances, whereas Predictive DG--SAV is cheaper below that level. SAV/CN remains near $O(1)$ error until $\Delta t\le2\times10^{-4}$ but is cost-competitive once it converges. Thus, on this benchmark, the DG--SAV split trades some of the raw speed of classical SAV for Full DG-level accuracy and a substantially smaller SAV gap.
}

\begin{figure}[!htbp]
\centering
\includegraphics[width=0.65\textwidth]{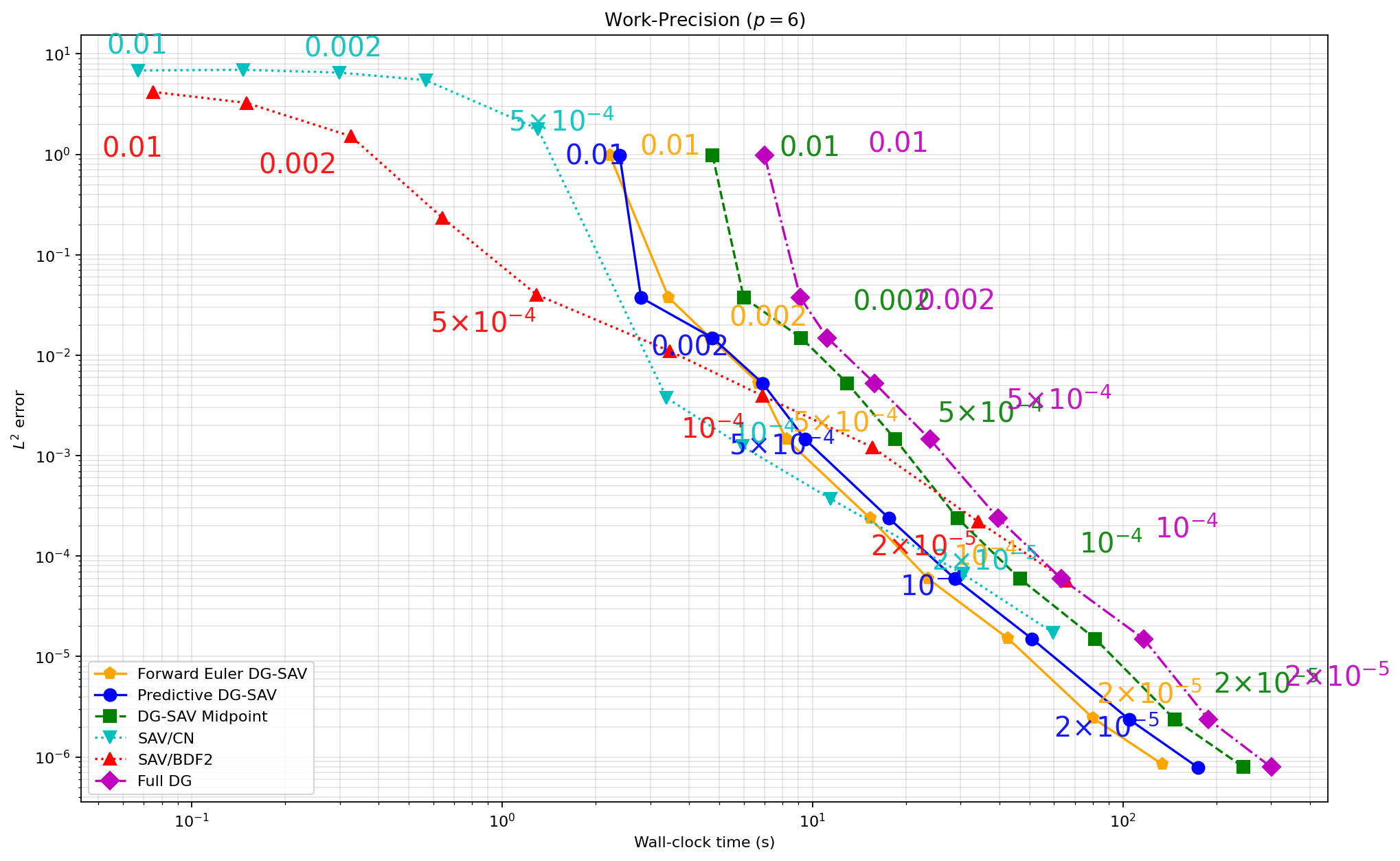}
\caption{Work--precision diagram for the $p=6$ benchmark ($T=1$), with the $\Delta t$ grid extended to $10^{-5}$. SAV/BDF2 crosses Predictive DG--SAV near an $L^2$ error of $3.4\times10^{-3}$, while Full DG is typically about twice as expensive as Predictive DG--SAV at matched accuracy. Times are solver-only means over repeated runs after one warm-up run.}
\label{fig:p6-workprecision}
\end{figure}

\subsection{Hamiltonian ODE: Double-Well Oscillator}
\label{sec:exp-doublewell}

{
This example illustrates the explicit regime of \cref{sec:separable}. For a separable Hamiltonian with quadratic kinetic energy, the natural split yields fully explicit Forward Euler and Predictive DG--SAV schemes.
}

\paragraph{{Problem and DG--SAV splitting}}
{
Consider the double-well oscillator $H(q,p)=\tfrac12p^2+\tfrac14(q^2-1)^2$, with $\Hdg(p)=\tfrac12p^2$, $\Hsav(q)=\tfrac14(q^2-1)^2$, and $C_0=1$. Under the paper-wide normalization $r=\sqrt{\Hsav(q)+C_0}$, the modified Hamiltonian $\Hhat(q,p,r)=\tfrac12p^2+r^2$ equals $H+C_0$ on the constraint manifold. Because $\Hdg$ is quadratic, its discrete gradient is the midpoint average in $p$. The Forward Euler and Predictive DG--SAV schemes of \cref{sec:separable} are therefore fully explicit and use the triangular update $r^{n+1}\to p^{n+1}\to q^{n+1}$. Both preserve $\Hhat$ exactly by \cref{cor:conservation-dissipation}.
}

\paragraph{{Experimental setup}}
{
We compare Forward Euler, Predictive, and Midpoint DG--SAV with Implicit Midpoint (a symplectic baseline) and Full DG (which preserves the original Hamiltonian). For \cref{tab:doublewell}, we use $T=10$, $h=0.05$, and five initial conditions, with an RK45 reference computed at tolerance $10^{-13}$. For each initial condition, the reported maximum errors are taken over the full trajectory; the table then averages those maxima, the final global errors, and the costs over the five cases. The long-time phase portrait and energy diagnostics in \cref{fig:dw_phase} use $(q^0,p^0)=(0,1.5)$, $T=50$, and $h=0.1$. Appendix~\ref{sec:supp-dw-convergence} reports a separate step-size study of the maximum original-energy defect.
}

\begin{table}[!htbp]
\centering
\caption{Double-well oscillator: averaged errors and cost over five initial conditions ($T=10$, $h=0.05$). Wall-clock values are solver-only arithmetic means after one warm-up (25 timed repetitions per method and initial condition).}
\label{tab:doublewell}
\resizebox{\textwidth}{!}{%
\begin{tabular}{lcccc}
\toprule
\textbf{Method} & \textbf{Max Mod.\ Err.} & \textbf{Max Original Err.} & \textbf{Endpoint Err.} & \textbf{Time/Step ($\mu$s)} \\
\midrule
Forward Euler DG-SAV & $2.84\times10^{-15}$ & $4.48\times10^{-1}$ & $8.13\times10^{-1}$ & $2.3$ \\
Predictive DG-SAV & $5.15\times10^{-15}$ & $1.08\times10^{-3}$ & $3.13\times10^{-3}$ & $2.3$ \\
Midpoint DG-SAV & $2.80\times10^{-15}$ & $6.28\times10^{-4}$ & $5.79\times10^{-3}$ & $183.0$ \\
Implicit Midpoint & --- & $6.99\times10^{-4}$ & $5.70\times10^{-3}$ & $98.2$ \\
Full DG & --- & $1.39\times10^{-15}$ & $5.22\times10^{-3}$ & $80.5$ \\
\bottomrule
\end{tabular}%
}
\end{table}

\begin{figure}[!htbp]
  \centering
  \includegraphics[width=0.80\textwidth]{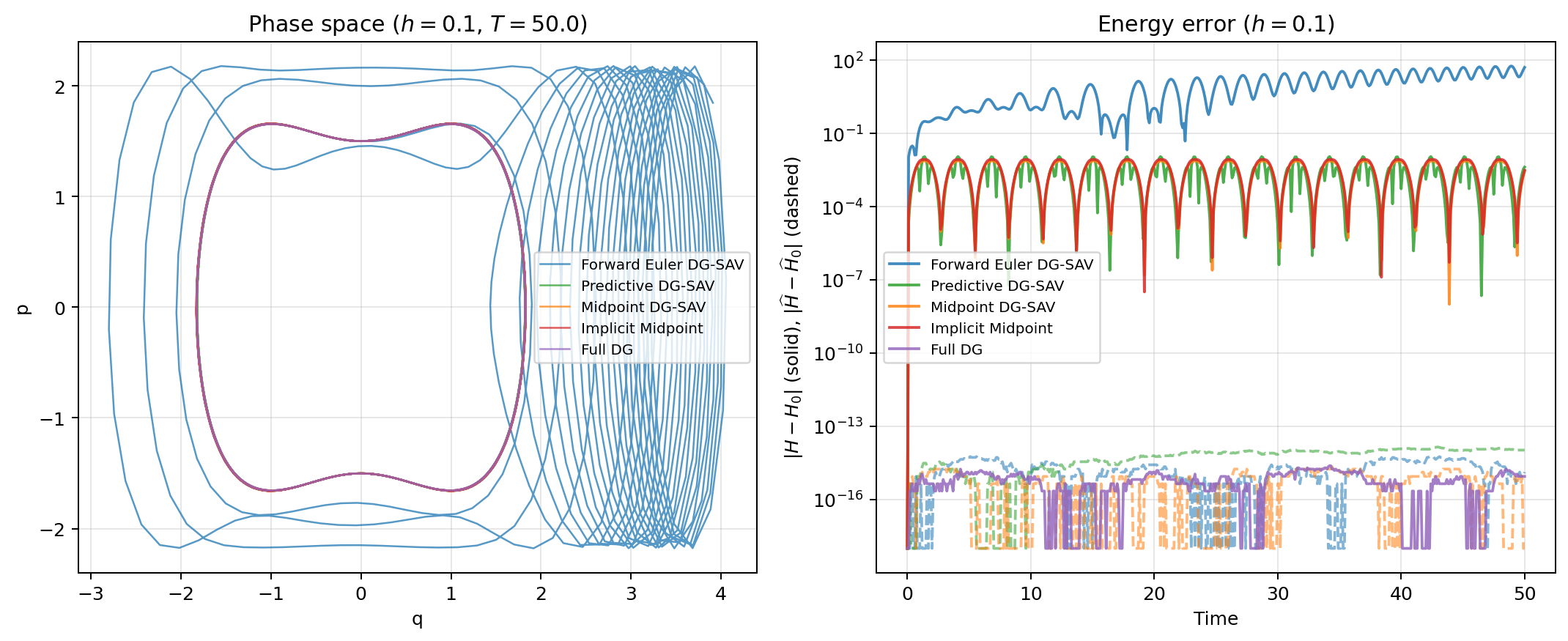}
  \caption{Double-well oscillator: phase portrait and energy behavior for representative step sizes.}
  \label{fig:dw_phase}
\end{figure}

\paragraph{{Discussion}}
{
The quadratic kinetic energy makes Forward Euler and Predictive DG--SAV fully explicit, with per-step costs of about $2.3\,\mu$s each, compared with $80$--$183\,\mu$s for the implicit methods in this implementation. Every DG--SAV variant conserves $\Hhat$ to machine precision. The original-energy error scales as $O(h)$ for Forward Euler and $O(h^2)$ for Predictive, with the latter comparable to implicit Midpoint DG--SAV. On this benchmark, Predictive DG--SAV combines second-order accuracy and exact modified-energy conservation with a speedup of roughly $80\times$ over Midpoint DG--SAV; Full DG preserves $H$ exactly but at a higher per-step cost.
}

\subsection{Poisson ODE with Nonlinear Casimir: Projected Tensor Construction}
\label{sec:exp-projected-casimir}

{
This example tests \cref{thm:projected-B}. For a nonlinear Casimir, the midpoint tensor $B(\bar z)$ need not satisfy the discrete Casimir condition; the projected tensor provides an explicit way to enforce it and recover exact Casimir preservation.
}

\paragraph{{Problem and methods}}
We construct a manufactured Poisson system in $\R^3$ with cubic Casimir $C(x) = x_1^3 + x_2^3 + x_3^3$ and Poisson tensor $B(x)\,v = \nabla C(x) \times v$, so that $B(x)\,\nabla C(x) = 0$ identically; the dynamics $\dot{x} = B(x)\,\nabla H(x)$ conserve~$C$ for any~$H$. We take the mixed quadratic-quartic Hamiltonian $H(x) = \tfrac{1}{2}\|x\|^2 + \tfrac{\beta}{4}\|x\|^4$ with $\beta = 1$, and the natural DG--SAV splitting $\Hdg(x) = \tfrac{1}{2}\|x\|^2$ (quadratic; AVF discrete gradient is the midpoint average) and $\Hsav(x) = \tfrac{\beta}{4}\|x\|^4$ (quartic; $C_0 = 1$). The exact componentwise discrete gradient of~$C$ is $(\widetilde{\nabla} C)_j = (x_j^+)^2 + x_j^+ x_j^- + (x_j^-)^2$, which reduces to $3x_j^2$ on the diagonal but does \emph{not} satisfy the discrete Casimir condition: $B(\bar{x})\,\widetilde{\nabla} C(x^+, x^-) \neq 0$ in general because $\nabla C(\bar x)$ and $\widetilde{\nabla} C$ differ by $O(h^2)$ terms that break the exact annihilation. Following \cref{thm:projected-B}, we restore the condition with the projected tensor $B_{\mathrm{proj}} = P\,B(\bar{x})\,P$, where $P = I - (g \otimes g)/\|g\|^2$ and $g = \widetilde{\nabla} C(x^+,x^-)$, which annihilates $g$ by construction. We compare four integrators: Implicit Midpoint with $B(\bar{x})$ and continuous gradients; DG--SAV Unprojected (discrete gradients, unprojected $B(\bar{x})$); DG--SAV Projected ($B_{\mathrm{proj}}$); and Full DG Projected. Implicit systems use \texttt{scipy.optimize.root} at tolerance $10^{-14}$, with initial condition $x^0 = (1.296, 0.648, 0.389)^\top$ ($\|x^0\| = 1.5$), $T = 1000$, and $h = 0.005$ (200{,}000 steps).

\begin{table}[!htbp]
\centering
\caption{Projected tensor example: conservation errors ($T=1000$, $h=0.005$, 200,000 steps).}
\label{tab:projected-casimir}
\begin{tabular}{lccc}
\toprule
\textbf{Method} & \textbf{Max $|\Delta H|$} & \textbf{Max $|\Delta\Hhat|$} & \textbf{Max $|\Delta C|$} \\
\midrule
Implicit Midpoint       & $1.5 \times 10^{-13}$ & N/A                    & $1.2 \times 10^{-3}$ \\
DG--SAV Unprojected     & $1.2 \times 10^{-13}$ & $3.6 \times 10^{-14}$  & $1.2 \times 10^{-3}$ \\
DG--SAV Projected       & $1.3 \times 10^{-13}$ & $6.0 \times 10^{-14}$  & $3.9 \times 10^{-13}$ \\
Full DG Projected       & $2.1 \times 10^{-13}$ & N/A                    & $2.7 \times 10^{-13}$ \\
\bottomrule
\end{tabular}
\end{table}

\begin{figure}[!htbp]
  \centering
  \includegraphics[width=0.95\textwidth]{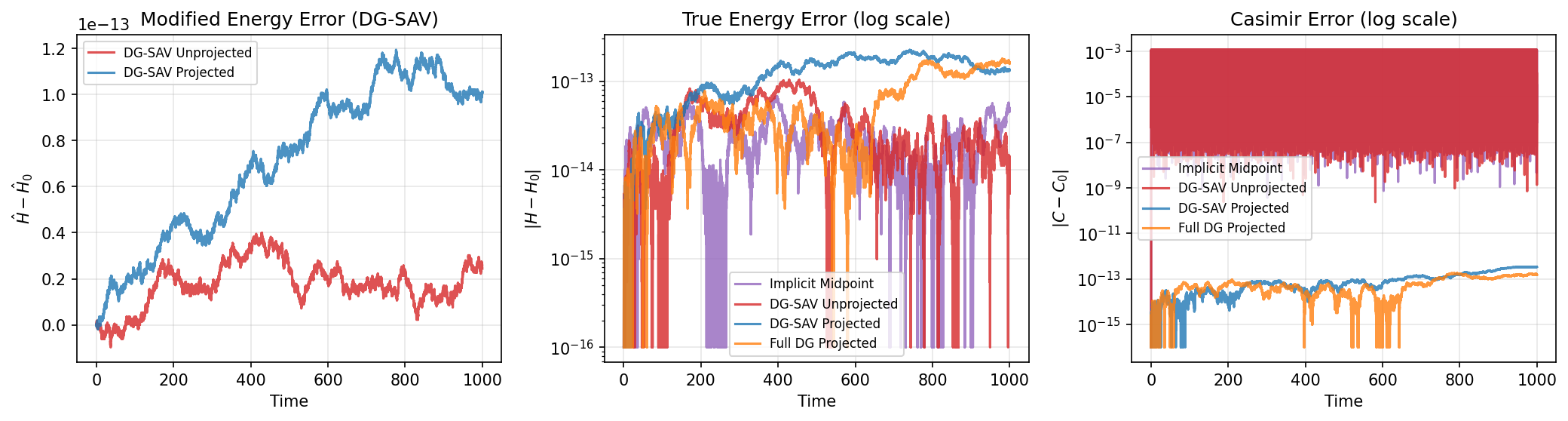}
  \caption{Projected tensor for Casimir preservation over $T=1000$ (200,000 steps).
  Left: modified energy error (DG--SAV variants only).
  Center: original energy error (log scale).
  Right: Casimir error (log scale).
  The unprojected methods exhibit bounded $O(10^{-3})$ Casimir drift, while the projected methods preserve $C$ to machine precision.}
  \label{fig:projected-casimir}
\end{figure}

\paragraph{{Discussion}}
{
The unprojected methods exhibit bounded $O(10^{-3})$ Casimir drift despite preserving energy to machine precision. The projected tensor enforces $B_{\mathrm{proj}}\widetilde{\nabla}C=0$ and reduces the Casimir error to approximately $10^{-13}$ without affecting energy conservation. The roundoff-level $|\Delta H|$ for all four methods is a consequence of the radial Hamiltonian $H=f(\|x\|^2)$ used in this test; for a generic quartic Hamiltonian, Implicit Midpoint instead exhibits the expected $O(h^2)$ oscillation in $H$. For Lie--Poisson systems with quadratic Casimirs, such as the rigid body, the midpoint tensor already satisfies the discrete condition. For nonlinear Casimirs, the projection in \cref{sec:casimir-projection} is one direct enforcement mechanism.
}

\paragraph{{Code availability}}
The source code and data used to generate the numerical results are being prepared for public release and will be made available with a subsequent version of this preprint.

\section{Conclusions and Future Research}
\label{sec:conclusions}

{
We have developed a DG--SAV framework for dissipative, Hamiltonian, and Poisson systems based on one extended discrete-gradient identity. The user-selected split $E=\Edg+\Esav$ interpolates between pure DG, which preserves the original energy law but may require a nonlinear solve, and SAV-type quadratization, which preserves a modified energy through cheaper updates. The three variants expose different accuracy--cost choices, while the projected Poisson construction also enforces selected Casimirs.

Natural extensions include adaptive or data-driven selection of the $\Edg/\Esav$ split; fourth- and higher-order methods based on symmetric composition or deferred correction; energy-aware step-size control; and applications to nonholonomic and other constrained systems with nonlinear potentials. The last direction can build on DG-based energy-preserving integrators~\cite{celledoni-Farre-Hoel-DMdD2019} and the broader nonholonomic literature surveyed in~\cite{ModinVerdier2017}.
}

\section*{Acknowledgments}

DMdD and MV acknowledge financial support from the Spanish Ministry of Science and Innovation under grants PID2022-137909NB-C21, PCI2024-155047-2 and from the Severo Ochoa Programme for Centres of Excellence in R\&D (CEX2023-001347-S).

\paragraph{{Use of artificial intelligence}}
{Generative AI tools were used to review and improve the writing of the manuscript and to assist with code for the numerical experiments. The authors take full responsibility for the content of the manuscript, including all mathematical statements, proofs, and numerical results.}

\clearpage
\appendix
\numberwithin{figure}{section}
\numberwithin{table}{section}

\begin{center}
\textbf{Supplementary Appendices}
\end{center}

\begin{center}
\textbf{Appendix contents}
\end{center}
\begingroup\small
\noindent\begin{tabular}{@{}p{\textwidth}@{}}
\hyperref[sec:supp-sav-discrete]{\ref*{sec:supp-sav-discrete} Explicit SAV discretization formulas}\\
\hyperref[sec:supp-numerics]{\ref*{sec:supp-numerics} Auxiliary numerical experiments}\\
\quad\hyperref[sec:supp-dw-convergence]{\ref*{sec:supp-dw-convergence} Double-well: maximum original-energy defect vs.\ step size}\\
\quad\hyperref[sec:supp-polynomial-ac]{\ref*{sec:supp-polynomial-ac} Polynomial Allen--Cahn: closed-form AVF regime}\\
\hyperref[sec:supp-proof-details]{\ref*{sec:supp-proof-details} Proof details: expansions of compressed steps}\\
\quad\hyperref[sec:supp-proof-theorem-44]{\ref*{sec:supp-proof-theorem-44} Comprehensive proof of Theorem~\ref*{thm:variants-accuracy}}\\
\quad\hyperref[sec:supp-proof-proposition-45]{\ref*{sec:supp-proof-proposition-45} Comprehensive proof of Proposition~\ref*{prop:well-posedness-diss}}\\
\quad\hyperref[sec:supp-proof-projected]{\ref*{sec:supp-proof-projected} Projected-tensor symmetry}
\end{tabular}
\endgroup

The appendices contain: (i) the standard SAV/CN and SAV/BDF2 formulas referenced in Section~\ref{sec:preliminaries} of the main text; (ii) two auxiliary experiments---a step-size study for the double-well oscillator and the polynomial Allen--Cahn closed-form-AVF regime; and (iii) detailed proofs of Theorem~\ref{thm:variants-accuracy} and Proposition~\ref{prop:well-posedness-diss}, together with the symmetry check for the projected tensor.

\section{Explicit SAV discretization formulas}
\label{sec:supp-sav-discrete}

{
We list the two standard second-order SAV discretizations referenced in Section~\ref{sec:preliminaries} of the main text. They are applied to $\dot x=-\nabla E(x)$ with $E=E_{\mathrm{lin}}+E_{\mathrm{nl}}$ and $r=\sqrt{E_{\mathrm{nl}}(x)+C_0}$.
}

\paragraph{SAV/CN (Crank--Nicolson)}
With $\bar x = \tfrac{1}{2}(x^{n+1}+x^n)$, an explicit $O(h^2)$ extrapolant $\bar x^* = \tfrac{1}{2}(3x^n - x^{n-1})$, and $\bar \Sigma := \sqrt{E_{\mathrm{nl}}(\bar x^*)+C_0}$, $\bar G_{\mathrm{nl}} := \nabla E_{\mathrm{nl}}(\bar x^*)$:
\begin{equation}\label{eq:supp-sav-cn}
\begin{aligned}
  \frac{x^{n+1} - x^n}{h} &= -\nabla E_{\mathrm{lin}}(\bar x) - \frac{r^{n+1}+r^n}{2\bar \Sigma}\,\bar G_{\mathrm{nl}}, \\[4pt]
  r^{n+1} - r^n &= \frac{\ip{\bar G_{\mathrm{nl}}}{x^{n+1}-x^n}}{2\bar \Sigma}.
\end{aligned}
\end{equation}
This scheme dissipates the single-level modified energy $\widehat{E}_{\mathrm{CN}}(x,r) = E_{\mathrm{lin}}(x) + r^2$ unconditionally in the time step~\cite[Theorem~2.1]{ShenXuYang2019}.

\paragraph{SAV/BDF2}
With the second-order extrapolant $x^* = 2x^n - x^{n-1}$, $\Sigma^* := \sqrt{E_{\mathrm{nl}}(x^*)+C_0}$, $G^*_{\mathrm{nl}} := \nabla E_{\mathrm{nl}}(x^*)$:
\begin{equation}\label{eq:supp-sav-bdf2}
\begin{aligned}
  \frac{3x^{n+1} - 4x^n + x^{n-1}}{2h} &= -\nabla E_{\mathrm{lin}}(x^{n+1}) - \frac{r^{n+1}}{\Sigma^*}\,G^*_{\mathrm{nl}}, \\[4pt]
  3r^{n+1} - 4r^n + r^{n-1} &= \frac{\ip{G^*_{\mathrm{nl}}}{3x^{n+1} - 4x^n + x^{n-1}}}{2\Sigma^*}.
\end{aligned}
\end{equation}
This scheme dissipates the two-level modified energy
\[
  \widehat{E}_{\mathrm{BDF2}}^{n+1,n} := \tfrac{1}{2}\bigl(E_{\mathrm{lin}}(x^{n+1}) + E_{\mathrm{lin}}(2x^{n+1}-x^n)\bigr) + \tfrac{1}{2}\bigl((r^{n+1})^2 + (2r^{n+1}-r^n)^2\bigr)
\]
unconditionally in the time step~\cite[Theorem~2.2]{ShenXuYang2019}. In both schemes, $\nabla E_{\mathrm{lin}}$ is linear (e.g., $\nabla E_{\mathrm{lin}}(x) = Ax$ for a constant SPSD matrix $A$), so the $x$-equation reduces to a linear system with \emph{constant coefficients}---the computational hallmark of SAV.

\section{Auxiliary numerical experiments}
\label{sec:supp-numerics}

\subsection{Double-well: maximum original-energy defect vs.\ step size}
\label{sec:supp-dw-convergence}

{
For the double-well oscillator in Section~\ref{sec:exp-doublewell} of the main text, we compute $\max_t|H(t)-H(0)|$ for $h\in[10^{-3},10^{-0.5}]$, using the barrier-crossing initial condition $(q^0,p^0)=(0,1.5)$ and $T=10$. The curves are consistent with $O(h)$ scaling for Forward Euler DG--SAV and $O(h^2)$ scaling for the Predictive and Midpoint variants.
}

\begin{figure}[H]
  \centering
  \includegraphics[width=0.65\textwidth]{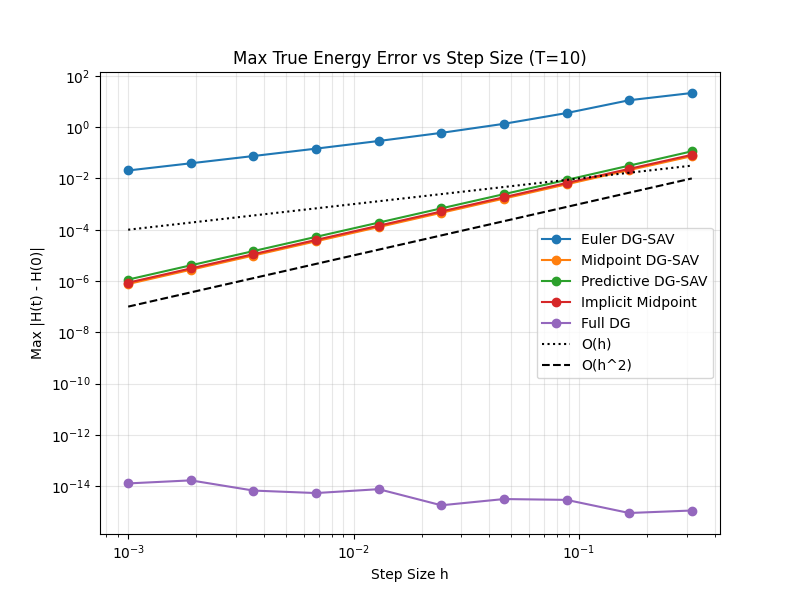}
  \caption{Double-well oscillator: maximum original-energy defect versus step size $h$ on a log--log scale.}
  \label{fig:supp-dw-energy-vs-h}
\end{figure}

\subsection{Polynomial Allen--Cahn: closed-form AVF regime}
\label{sec:supp-polynomial-ac}\label{sec:supp-ac-cost}

{
The polynomial Allen--Cahn equation provides a complementary regime to the logarithmic Flory--Huggins benchmark in Section~\ref{sec:num-allen-cahn} of the main text:
}
$u_t = \varepsilon^2 \Delta u - (u^3 - u)$,
{
with $E(u)=\int_\Omega[\tfrac{\varepsilon^2}{2}|\nabla u|^2+\tfrac14(u^2-1)^2]\,dx$. Since the bulk potential is quartic, its AVF discrete gradient has the pointwise closed form
\[
\widetilde\nabla_{\mathrm{AVF}}\Esav(v,w)
=\tfrac14(v^3+v^2w+vw^2+w^3)-\tfrac12(v+w),
\]
which requires neither inner integration nor transcendental functions. Full DG and Midpoint DG--SAV therefore have comparable Picard-iteration costs and accuracy; the non-iterative Predictive scheme is the principal lower-cost DG--SAV alternative. \Cref{tab:supp-polynomial-ac-cost} reports per-step diagnostics for $\Delta t=10^{-2}$, $N=128$, $\varepsilon=0.1$, $T=1$, and $u_0=\tanh((x-\pi)/\varepsilon)+0.1\sin(4x)$.
}

\begin{table}[H]
\centering
\caption{Polynomial Allen--Cahn: per-step cost and SAV gap ($\Delta t = 10^{-2}$, $T=1$). Full DG and Midpoint DG--SAV are essentially equivalent because the polynomial AVF is closed-form.}\label{tab:supp-polynomial-ac-cost}
\begin{tabular}{lcccc}
\toprule
& Iters/step & Time/step ($\mu$s) & SAV gap & Obs.\ order \\
\midrule
Forward Euler DG--SAV & 1.0  & 92  & $2.4\times 10^{-3}$ & 1 \\
Predictive DG--SAV & 1.0  & 120 & $3.2\times 10^{-4}$ & 2 \\
Midpoint DG--SAV   & 5.3  & 305 & $7.4\times 10^{-5}$ & 2 \\
SAV/CN             & 1.0  & 97  & $1.1\times 10^{-3}$ & 2 \\
SAV/BDF2           & 1.0  & 105 & $2.6\times 10^{-3}$ & 2 \\
Full DG            & 5.3  & 339 & 0 & 2 \\
\bottomrule
\end{tabular}
\end{table}

{
In this closed-form-AVF regime, DG--SAV has no cost advantage over Full DG unless the non-iterative Predictive variant is used. The intermediate accuracy--cost regime studied in the main text instead requires an AVF that is more expensive than one gradient evaluation, as occurs for the logarithmic Flory--Huggins potential and the nonlocal $p=6$ term. The polynomial double well does not have that property.
}

\section{Proof details: expansions of compressed steps}
\label{sec:supp-proof-details}

The proofs in the main text are intentionally compact. This appendix gives standalone details for the two main local-analysis results, using the same notation as the main text. Throughout, $\Sigma(x)=\sqrt{\Esav(x)+C_0}$ and $G(x)=\nabla\Esav(x)$, and the positivity assumption $\Esav+C_0\ge\delta>0$ is used to make $\Sigma^{-1}$ and all SAV coefficient ratios smooth and locally bounded.

\subsection{Comprehensive proof of Theorem~\ref{thm:variants-accuracy}}
\label{sec:supp-proof-theorem-44}

We prove the order statements by local truncation error estimates on the extended variable $y=(x,r)$, followed by the standard local-to-global argument for locally well-posed one-step methods. The exact extended solution satisfies the constraint $r(t)=\Sigma(x(t))$ and the smooth ODE
\[
  \dot x = S(x)\nabla E(x), \qquad
  \dot r = \frac{\ip{G(x)}{\dot x}}{2\Sigma(x)}.
\]
On a compact tube around the exact trajectory, the assumptions of the theorem imply that $S$, $G$, $\Sigma^{-1}$, and the relevant discrete-gradient maps are locally Lipschitz. Consequently, once a one-step defect is $O(h^{p+1})$, the usual stability estimate for one-step methods gives a global $O(h^p)$ error on fixed time intervals.

\paragraph{Preliminary midpoint expansions}
Let $x_m=x(t_n+h/2)$, $r_m=r(t_n+h/2)$, and let $\Delta x=x(t_{n+1})-x(t_n)$. Taylor expansion about $t_n+h/2$ gives
\begin{equation}\label{eq:supp-midpoint-taylor-xr}
  x(t_{n+1})-x(t_n) = h\dot x(t_n+h/2)+O(h^3), \qquad
  r(t_{n+1})-r(t_n) = h\dot r(t_n+h/2)+O(h^3),
\end{equation}
and
\begin{equation}\label{eq:supp-midpoint-averages}
  \tfrac12\bigl(x(t_{n+1})+x(t_n)\bigr)=x_m+O(h^2), \qquad
  \tfrac12\bigl(r(t_{n+1})+r(t_n)\bigr)=r_m+O(h^2).
\end{equation}
Because $r_m=\Sigma(x_m)$ and $\Sigma$ is bounded below by $\sqrt{\delta}$, any coefficient evaluation that is $O(h^2)$-close to $x_m$ yields
\begin{equation}\label{eq:supp-ratio-midpoint}
  \frac{\bar r}{\Sigma_*}=1+O(h^2)
\end{equation}
when $\bar r=\tfrac12(r(t_{n+1})+r(t_n))$ is computed from the exact solution.

The symmetric discrete gradient provides the second preliminary estimate. By Lemma~\ref{lem:symmDG-midpoint}, for exact endpoints,
\begin{equation}\label{eq:supp-dg-midpoint-estimate}
  \widetilde{\nabla}\Edg(x(t_{n+1}),x(t_n))
  =
  \nabla\Edg\!\left(\tfrac12(x(t_{n+1})+x(t_n))\right)+O(h^2)
  =
  \nabla\Edg(x_m)+O(h^2).
\end{equation}
For the Forward Euler variant only the weaker consequence
$\widetilde{\nabla}\Edg(x(t_{n+1}),x(t_n))=\nabla\Edg(x(t_n))+O(h)$
is needed.

\paragraph{Forward Euler: first order}
For Forward Euler, $x_*=x(t_n)$, $S_*=S(x(t_n))$, $\Sigma_*=\Sigma(x(t_n))$, and $G_*=G(x(t_n))$. Evaluating the method residual on exact data gives
\[
  \bar r = r(t_n)+O(h), \qquad
  \frac{\bar r}{\Sigma_*}=1+O(h),
\]
and, as noted above,
\[
  \widetilde{\nabla}\Edg(x(t_{n+1}),x(t_n))
  =
  \nabla\Edg(x(t_n))+O(h).
\]
Hence the effective $x$-gradient in the scheme satisfies
\[
\begin{aligned}
  \eta_* &{}=
  \widetilde{\nabla}\Edg(x(t_{n+1}),x(t_n))
  + \frac{\bar r}{\Sigma_*}G_* \\
  &{}= \nabla\Edg(x(t_n))+G(x(t_n))+O(h) \\
  &{}=
  \nabla E(x(t_n))+O(h).
\end{aligned}
\]
Since $S_*=S(x(t_n))$, the right-hand side of the $x$-equation is
$S(x(t_n))\nabla E(x(t_n))+O(h)=\dot x(t_n)+O(h)$. Therefore
\[
  x(t_{n+1})-x(t_n)-hS_*\eta_*=O(h^2).
\]
For the scalar equation,
\[
  \frac{\ip{G_*}{x(t_{n+1})-x(t_n)}}{2\Sigma_*}
  =
  h\,\frac{\ip{G(x(t_n))}{\dot x(t_n)}}{2\Sigma(x(t_n))}+O(h^2)
  =
  h\dot r(t_n)+O(h^2),
\]
so $r(t_{n+1})-r(t_n)-\ip{G_*}{\Delta x}/(2\Sigma_*)=O(h^2)$.
Thus the local defect is $O(h^2)$, and the global error is $O(h)$.

\paragraph{Midpoint: self-adjointness}
For the midpoint variant, $x_*=\bar x=\tfrac12(x^{n+1}+x^n)$ and $\bar r=\tfrac12(r^{n+1}+r^n)$. Interchange the two time levels and replace $h$ by $-h$. The quantities $\bar x$, $\bar r$, $S(\bar x)$, $\Sigma(\bar x)$, and $G(\bar x)$ are unchanged, while $x^{n+1}-x^n$ and $r^{n+1}-r^n$ change sign. The symmetric discrete gradient is invariant under the same interchange:
\[
  \widetilde{\nabla}\Edg(x^{n+1},x^n)
  =
  \widetilde{\nabla}\Edg(x^n,x^{n+1}).
\]
Therefore both equations of the midpoint method are transformed into themselves under the reversal $(x^n,r^n,h)\leftrightarrow(x^{n+1},r^{n+1},-h)$. The one-step map is self-adjoint, $\Phi_{-h}=\Phi_h^{-1}$.

\paragraph{Midpoint: second order}
We also verify the local defect directly. With exact endpoints, $\bar x=x_m+O(h^2)$ and $\bar r=r_m+O(h^2)$ by~\eqref{eq:supp-midpoint-averages}. Smoothness gives
\[
  S(\bar x)=S(x_m)+O(h^2), \quad
  G(\bar x)=G(x_m)+O(h^2), \quad
  \Sigma(\bar x)=\Sigma(x_m)+O(h^2),
\]
and \eqref{eq:supp-ratio-midpoint} gives $\bar r/\Sigma(\bar x)=1+O(h^2)$. Combining this with~\eqref{eq:supp-dg-midpoint-estimate},
\[
\begin{aligned}
  \eta_* &{}=
  \widetilde{\nabla}\Edg(x(t_{n+1}),x(t_n))
  +\frac{\bar r}{\Sigma(\bar x)}G(\bar x) \\
  &{}= \nabla\Edg(x_m)+G(x_m)+O(h^2) \\
  &{}=
  \nabla E(x_m)+O(h^2).
\end{aligned}
\]
Thus
\[
  hS(\bar x)\eta_*
  =
  hS(x_m)\nabla E(x_m)+O(h^3)
  =
  h\dot x(t_n+h/2)+O(h^3).
\]
The $x$-equation residual is $O(h^3)$ by~\eqref{eq:supp-midpoint-taylor-xr}. For the $r$-equation,
\[
  \frac{\ip{G(\bar x)}{x(t_{n+1})-x(t_n)}}{2\Sigma(\bar x)}
  =
  h\,\frac{\ip{G(x_m)}{\dot x(t_n+h/2)}}{2\Sigma(x_m)}+O(h^3)
  =
  h\dot r(t_n+h/2)+O(h^3),
\]
again by~\eqref{eq:supp-midpoint-taylor-xr}. Hence the full local defect is $O(h^3)$ and the global error is $O(h^2)$. This direct estimate is consistent with the standard symmetry/order-doubling argument: the method is self-adjoint and consistent, so its order is even; the first possible order is therefore two.

\paragraph{Predictive: predictor accuracy}
For the Predictive variant, the coefficient point is
\[
  x_{n+1/2}^{\mathrm{pred}}
  =
  x(t_n)+\frac h2 S(x(t_n))
  \left(\nabla\Edg(x(t_n))+\frac{r(t_n)}{\Sigma(x(t_n))}G(x(t_n))\right).
\]
On the constraint manifold $r(t_n)=\Sigma(x(t_n))$, the quantity in parentheses is $\nabla E(x(t_n))$, so
\[
  x_{n+1/2}^{\mathrm{pred}}
  =
  x(t_n)+\frac h2\dot x(t_n)
  =
  x_m+O(h^2).
\]
Consequently,
\begin{equation}\label{eq:supp-pred-coeff-accuracy}
  S_{\mathrm{pred}}=S(x_m)+O(h^2),\qquad
  G_{\mathrm{pred}}=G(x_m)+O(h^2),\qquad
  \Sigma_{\mathrm{pred}}=\Sigma(x_m)+O(h^2).
\end{equation}
Together with~\eqref{eq:supp-midpoint-averages}, this gives
\[
  \frac{\bar r}{\Sigma_{\mathrm{pred}}}
  =
  \frac{r_m+O(h^2)}{\Sigma(x_m)+O(h^2)}
  =
  1+O(h^2).
\]

\paragraph{Predictive: local defect}
Using~\eqref{eq:supp-dg-midpoint-estimate} and~\eqref{eq:supp-pred-coeff-accuracy},
\[
\begin{aligned}
  \eta_* &{}=
  \widetilde{\nabla}\Edg(x(t_{n+1}),x(t_n))
  + \frac{\bar r}{\Sigma_{\mathrm{pred}}}G_{\mathrm{pred}} \\
  &{}= \nabla\Edg(x_m)+G(x_m)+O(h^2) \\
  &{}=
  \nabla E(x_m)+O(h^2).
\end{aligned}
\]
Therefore
\[
  hS_{\mathrm{pred}}\eta_*
  =
  hS(x_m)\nabla E(x_m)+O(h^3)
  =
  h\dot x(t_n+h/2)+O(h^3),
\]
and the $x$-equation residual on exact endpoints is $O(h^3)$. Similarly,
\[
  \frac{\ip{G_{\mathrm{pred}}}{x(t_{n+1})-x(t_n)}}{2\Sigma_{\mathrm{pred}}}
  =
  h\,\frac{\ip{G(x_m)}{\dot x(t_n+h/2)}}{2\Sigma(x_m)}+O(h^3)
  =
  h\dot r(t_n+h/2)+O(h^3),
\]
{
so the $r$-equation residual is also $O(h^3)$. The contraction estimate underlying Proposition~\ref{prop:well-posedness-diss} gives local stability of the frozen-coefficient corrector for sufficiently small $h$. Applying that stability estimate to the two residuals yields an $O(h^3)$ numerical one-step defect in $(x,r)$ and hence second-order global accuracy.
}

\paragraph{Conservative structures}
The preceding consistency estimates use smoothness of $S(\cdot)$, smoothness and boundedness of $\Sigma^{-1}$, and the symmetry/midpoint consistency of the discrete gradient. They do not use negative semidefiniteness of $S$. Thus the same local truncation arguments apply verbatim when $S(x)=J(x)$ is skew-symmetric in the Hamiltonian or Poisson setting. Skew-symmetry changes the sign structure of the energy identity but not the order calculation.

\subsection{Comprehensive proof of Proposition~\ref{prop:well-posedness-diss}}
\label{sec:supp-proof-proposition-45}

For the Forward Euler and Predictive variants the coefficient triple $(S_*,\Sigma_*,G_*)$ is fixed before the implicit corrector is solved. Let
\[
  B_R(x^n)=\{x\in\R^d:\|x-x^n\|\le R\}
\]
and assume the hypotheses of Proposition~\ref{prop:well-posedness-diss}. We prove that, for sufficiently small $h$, the corrector has a unique solution with $x^{n+1}\in B_R(x^n)$.

\paragraph{Eliminating the scalar unknown}
The second DG--SAV equation gives, for any candidate $x$,
\[
  r(x)=r^n+\frac{\ip{G_*}{x-x^n}}{2\Sigma_*},\qquad
  \bar r(x)=\frac{r(x)+r^n}{2}
  =
  r^n+\frac{\ip{G_*}{x-x^n}}{4\Sigma_*}.
\]
Since $\Sigma_*^2\ge\delta>0$, this affine map is well-defined and satisfies
\begin{equation}\label{eq:supp-barr-lip}
  |\bar r(x)-\bar r(y)|
  \le
  \frac{\|G_*\|}{4\Sigma_*}\|x-y\|.
\end{equation}
Substituting $\bar r(x)$ into the $x$-equation gives the fixed-point map
\begin{equation}\label{eq:supp-fixed-point-map}
  T(x)
  =
  x^n
  + hS_*
  \left(
    \widetilde{\nabla}\Edg(x,x^n)
    +\frac{\bar r(x)}{\Sigma_*}G_*
  \right).
\end{equation}
Fixed points of $T$ are exactly the $x$-components of solutions of the original two-equation corrector; once such an $x$ is known, $r^{n+1}=r(x)$ is uniquely determined.

\paragraph{Contraction estimate}
For $x,y\in B_R(x^n)$, the assumed Lipschitz bound on the discrete gradient and~\eqref{eq:supp-barr-lip} give
\[
\begin{aligned}
  \|T(x)-T(y)\|
  &\le
  h\|S_*\|
  \left(
    \|\widetilde{\nabla}\Edg(x,x^n)-\widetilde{\nabla}\Edg(y,x^n)\|
    +\frac{|\bar r(x)-\bar r(y)|}{\Sigma_*}\|G_*\|
  \right) \\
  &\le
  h\mu_S
  \left(
    L_{\mathrm{DG}}(R)
    +\frac{\|G_*\|^2}{4\Sigma_*^2}
  \right)\|x-y\| \\
  &\le
  hL_R\|x-y\|,
\end{aligned}
\]
where
\begin{equation}\label{eq:supp-LR}
  L_R:=\mu_S\left(L_{\mathrm{DG}}(R)+\frac{\|G_*\|^2}{4\delta}\right).
\end{equation}
Thus $T$ is a contraction on $B_R(x^n)$ whenever $hL_R<1$.

\paragraph{Invariance of the ball}
It remains to ensure that $T$ maps $B_R(x^n)$ into itself. At the center of the ball,
\[
  T(x^n)-x^n
  =
  hS_*
  \left(
    \widetilde{\nabla}\Edg(x^n,x^n)+\frac{r^n}{\Sigma_*}G_*
  \right),
\]
so define
\[
  \rho_0 :=
  \left\|
    S_*
    \left(
      \nabla\Edg(x^n)+\frac{r^n}{\Sigma_*}G_*
    \right)
  \right\|,
\]
using consistency of the discrete gradient at the diagonal. For any $x\in B_R(x^n)$,
\[
  \|T(x)-x^n\|
  \le
  \|T(x)-T(x^n)\|+\|T(x^n)-x^n\|
  \le
  hL_RR+h\rho_0.
\]
Hence $T(B_R(x^n))\subset B_R(x^n)$ provided
\[
  h(L_RR+\rho_0)\le R.
\]

\paragraph{Choice of the step-size threshold}
One explicit admissible threshold is
\begin{equation}\label{eq:supp-h0}
  h_0
  =
  \min\left\{
    \frac{1}{2L_R},
    \frac{R}{L_RR+\rho_0}
  \right\},
\end{equation}
with the convention that $1/(2L_R)=+\infty$ if $L_R=0$. For every $0<h<h_0$, the map $T$ is a strict contraction on the complete metric space $B_R(x^n)$ and maps that ball into itself. Banach's fixed-point theorem gives a unique fixed point $x^{n+1}\in B_R(x^n)$. The corresponding scalar value
\[
  r^{n+1}=r^n+\frac{\ip{G_*}{x^{n+1}-x^n}}{2\Sigma_*}
\]
is then unique. This proves local existence and uniqueness for the frozen-coefficient Forward Euler and Predictive variants. The proof uses only the norm bound on $S_*$, not a sign condition, so it covers both dissipative and conservative choices of the structure operator.

\subsection{Projected-tensor interchange symmetry}
\label{sec:supp-proof-projected}

{
Theorem~\ref{thm:poisson-midpoint} assumes that the two-point tensor is invariant under interchange of its arguments,
\[
\widetilde B(z^+,z^-)=\widetilde B(z^-,z^+).
\]
The two tensors used in the main text have this property:
\begin{itemize}
  \item For $\widetilde B=B(\bar z)$, the result follows from $\bar z=(z^++z^-)/2$, which is unchanged by the interchange.
  \item For $\widetilde B_P=P\widetilde B_0P$, each symmetric Casimir discrete gradient satisfies $\widetilde\nabla C_j(z^+,z^-)=\widetilde\nabla C_j(z^-,z^+)$. Consequently, the rectangular matrix $\mathcal C=[\widetilde\nabla C_1\ \cdots\ \widetilde\nabla C_k]$ and the projector $P=I-\mathcal C(\mathcal C^\top\mathcal C)^{-1}\mathcal C^\top$ are invariant under the interchange. If $\widetilde B_0$ has the same invariance, as it does for $\widetilde B_0=B(\bar z)$, then so does $\widetilde B_P$.
\end{itemize}
This verifies the interchange-invariance assumption of Theorem~\ref{thm:poisson-midpoint} for $\widetilde B=B(\bar z)$ in~\eqref{eq:poisson-midpoint} and for the projected variant used in Section~\ref{sec:exp-projected-casimir}. Skew-symmetry and consistency follow from Theorem~\ref{thm:projected-B}; the required regularity follows on regions where $\mathcal C^\top\mathcal C$ remains invertible.
}

\bibliographystyle{siamplain}
\bibliography{references}

\end{document}